\documentclass
[
    a4paper,
    DIV=11,
    abstracton,
    numbers=noendperiod
]
{scrartcl}

\usepackage
{
    graphicx,
    amssymb,
    amsmath,
    amsthm,
    dsfont, 
    epstopdf,
    braket,
    mathtools,
    algorithm,
    algpseudocode,
    authblk,
    tikz,
    kbordermatrix,
    enumitem,
    stmaryrd
}

\usepackage[nomessages]{fp}
\usepackage[bf,normal]{caption}
\usepackage[labelformat=simple]{subcaption}

\usepackage{graphicx}
\usepackage{color}

\newcommand{\subfiguretitle}[1]{{\scriptsize{#1}} \\[1mm]}
\newcommand{\R}{\mathbb{R}}                                     
\providecommand{\abs}[1]{\left\lvert #1 \right\rvert}           

\newcommand\xqed[1]{\leavevmode\unskip\penalty9999 \hbox{}\nobreak\hfill \quad\hbox{#1}}
\newcommand{\exampleSymbol}{\xqed{$\triangle$}}

\usetikzlibrary{shapes.geometric,shapes.misc,positioning,calc}

\makeatletter
\newcommand{\algrule}[1][0.4pt]{\par\vskip.3\baselineskip\hrule height #1\par\vskip.3\baselineskip}
\makeatother

\definecolor{YellowOrange}{RGB}{250,162,26}
\definecolor{Gray}{RGB}{148,150,152}

\usepackage[colorlinks,
            pdffitwindow=false,
            plainpages=false,
            pdfpagelabels=true,
            pdfpagemode=UseOutlines,
            pdfpagelayout=SinglePage,
            bookmarks=false,
            colorlinks=true,
            hyperfootnotes=false,
            linkcolor=blue,
            citecolor=green!50!black]{hyperref}

\DeclareMathAlphabet{\mathpzc}{OT1}{pzc}{m}{it}

\newtheorem{theorem}{Theorem}[section]

\newtheorem{lemma}[theorem]{Lemma}

\newtheorem{definition}[theorem]{Definition}
\theoremstyle{definition}
\newtheorem{example}[theorem]{Example}
\newtheorem{remark}[theorem]{Remark}

\providecommand{\keywords}[1]{\begin{addmargin}[2.5em]{2.5em}\noindent\textbf{Keywords:} #1\\\end{addmargin}}
\providecommand{\msc}[1]{\begin{addmargin}[2.5em]{2.5em}\noindent\textbf{2010 MSC:} #1\\\end{addmargin}}

\makeatletter
\renewcommand*\env@matrix[1][*\c@MaxMatrixCols c]{%
  \hskip -\arraycolsep
  \let\@ifnextchar\new@ifnextchar
  \array{#1}}
\makeatother

\newcommand{\mat}[3]{#1 \left| \begin{array}{@{}l@{}} \scriptstyle #3 \\[0cm] \scriptstyle #2 \end{array}  \right.}

\mathtoolsset{centercolon} 

\title{Multidimensional approximation of \\ nonlinear dynamical systems}
\author[1]{Patrick Gel\ss}
\author[1]{Stefan Klus}
\author[2]{Jens Eisert}
\author[1,3]{Christof Sch\"utte}
\affil[1]{\small Department of Mathematics and Computer Science, Freie Universit\"at Berlin, Germany}
\affil[2]{\small Dahlem Center for Complex Quantum Systems, Freie Universit\"at Berlin, Germany}
\affil[3]{\small Zuse Institute Berlin, Germany}
\date{}

\begin{document}
\maketitle


\begin{abstract}
A key task in the field of modeling and analyzing nonlinear dynamical systems is the recovery of unknown governing equations from measurement data only. There is a wide range of application areas for this important instance of system identification, ranging from industrial engineering and acoustic signal processing to stock market models. In order to find appropriate representations of underlying dynamical systems, various data-driven methods have been proposed by different communities. However, if the given data sets are high-dimensional, then these methods typically suffer from the curse of dimensionality. To significantly reduce the computational costs and storage consumption, we propose the method \emph{MANDy} which combines data-driven methods with tensor network decompositions. The efficiency of the introduced approach will be illustrated with the aid of several high-dimensional nonlinear dynamical systems.
\end{abstract}


\msc{15A69, 34C15, 62J05, 65E05, 93B30}

\section{Introduction}

The identification of governing equations from data plays an increasingly important role in the field of nonlinear dynamical systems. After all, in many practically relevant situations, the actual underlying model that captures a dynamical system is unknown, while the only information available is data arising from the natural time evolution. In this instance of system identification, given time-dependent measurement data, the task at hand is to reconstruct the underlying system, merely assuming that the governing equations can be accurately approximated by linear combinations of preselected basis functions. We consider autonomous dynamical systems given by ordinary differential equations (ODEs). Many physical --~even chaotic~-- systems have a simple ODE structure comprising only a small number of terms, e.g., monomials or trigonometric functions and combinations thereof. That is, if we choose a large set of basis functions, then a system may be described in terms of only a few of these basis functions, or -- in other words -- the governing equations are sparse in the high-dimensional space of selected basis functions, cf.~\cite{SPROTT1994, SOCOLAR2006, TESCHL2012}. Essentially, we are interested in the relationship between two data sets, namely measurements of $X(t)$ and measurements of $\dot{X}(t)$ at certain time points $t$. The methods we will discuss below (as well as the approach we propose), however, can be applied to a wider range of problems. For instance, we could also consider discrete dynamical systems or stochastic differential equations (SDEs). The so-called SINDy approach \cite{BRUNTON2016} has also already been extended to identify the governing equations of SDEs, see~\cite{BONINSEGNA2018}.

The challenge here is to recover the dynamics (i.e., to recover $F$ such that $\dot{X}(t) = F(X(t))$) from a rather small number of measurements of the time-dependent states $X(t)$ and the corresponding derivatives. This can be accomplished by using different approaches, see, for example, \cite{SJOBERG1995}. One of them is \emph{symbolic regression} \cite{KOZA1994}, where evolutionary algorithms reveal appropriate models that fit the given measurement data by combining mathematical expressions. Unfortunately, symbolic regression is in general too expensive and the convergence might be slow for high-dimensional data sets because discovering the governing equations from measurement data in a high-dimensional space $\mathbb{R}^d$ can be prohibitively expensive in terms of memory consumption and computational costs. This phenomenon is often referred to as the \emph{``curse of dimensionality''}. Also neural networks can be used to approximate the governing equations of dynamical systems (and, in fact, to approximate any multidimensional function), see~\cite{MUZHOU2011}. As for the symbolic regression, knowledge about the governing equations is not needed a priori for these purely data-driven methods. However, the approximation based on neural networks does not necessarily lead to interpretable representations of the governing equations since the given data is only interpolated by using artificial activation functions, e.g., radial basis functions. 

In~\cite{BRUNTON2016}, Brunton et al.\ proposed an algorithm for the \emph{sparse identification of nonlinear dynamical systems} (SINDy). Given data measurements of states and corresponding derivatives, a certain set of basis functions has to be chosen a priori. Based on sequential least-squares approximations and an iterative selection of relevant basis functions -- which is comparable to \emph{iterative hard thresholding} \cite{BLUMENSATH2012} -- the application of SINDy often results in a sparse representation of the underlying system, provided that we choose suitable basis functions. We could also add an $L^1$-regularization term to the regression problem if sparsity of the solution should be enforced. In the community of compressed sensing, this method is known as the \emph{least absolute shrinkage and selection operator} (LASSO), see~\cite{TIBSHIRANI1996}. However, for large data sets (particularly for large dimensions $d$) the above methods still suffer from the curse of dimensionality. Thus, there is a need for efficient methods that are able to recover the governing equations of high-dimensional nonlinear dynamical systems.

In this work, we combine the data-driven discovery of the underlying dynamics with the computation of pseudoinverses of tensor-network decompositions. Tensors are generalizations of vectors and matrices represented by multidimensional arrays with multiple indices. The interest in low-rank tensor decompositions has been growing rapidly over the last years since several tensor formats such as the canonical format \cite{CARROLL1970, HARSHMAN1970}, the Tucker format \cite{TUCKER1963,TUCKER1964}, and the hierarchical Tucker format \cite{HACKBUSCH2009, ARNOLD2013} have shown that it is possible to mitigate the curse of dimensionality for many high-dimensional problems. Tensor-based methods have been successfully applied to many different application areas, e.g., quantum physics \cite{WHITE1992, MEYER2009,Orus-AnnPhys-2014, VerstraeteBig, EisertTensors}, chemical reaction dynamics \cite{JAHNKE2008, GELSS2016}, machine learning \cite{BEYLKIN2009,NOVIKOV2015}, and high-dimensional data analysis \cite{KLUS2016, KLUS2018}. 

It turned out that one of the most promising tensor formats is the so-called tensor-train format (TT format) \cite{OSELEDETS2009a, OSELEDETS2009b, OSELEDETS2011} or, equivalently, the matrix product state (MPS) format \cite{MPSSurvey,MPSRev}. It is a special case of the hierarchical Tucker format and combines the advantages of the canonical format and the Tucker format, i.e., the storage consumption of a tensor train does not depend exponentially on the number of dimensions and there exist robust algorithms for the computation of best approximations. The fact that the TT format has emerged independently in several fields of science signifies its importance. In quantum physics, the MPS representation, on the one hand, dates back to work on the density-matrix renormalization group \cite{WHITE1992}. On the other hand, pure finitely correlated states \cite{raey} are again basically quantum states in the TT format, originating in mathematical physics in efforts to generalize notions of hidden Markov models. Within numerical mathematics, the use of tensor network decompositions and approximations is rather new, which has led to an exciting development. However, many questions are still open, specifically when it comes to the convergence when approximating solutions of systems of linear equations given in the TT format, cf.~\cite{HOLTZ2012}. For an overview of different low-rank tensor approximation approaches, we refer to~\cite{GRASEDYCK2013,GELSS2017b,Orus-AnnPhys-2014}. 

The new role of tensor decompositions for data analysis techniques and especially compressed sensing was already discussed in, e.g.,~\cite{KLUS2018, RAUHUT2015, CICHOCKI2015}. Using the TT format we shift the focus from \emph{sparse} structures to \emph{low-rank} or \emph{data-sparse structures}. That is, instead of determining sparse vectors or matrices, we aim at recovering dynamical systems by utilizing tensor decompositions with a moderate number of core elements. Applying the approach proposed in~\cite{KLUS2018} to construct pseudoinverses in the TT format, we can extend the method to identify governing equations so that low-rank representations of the data can be utilized to reduce the computational complexity as well as the memory requirements. Referring to SINDy, we call our method MANDy, which is short for \emph{multidimensional approximation of nonlinear dynamical systems}.

Data-driven discovery of dynamical systems will continue to play an increasingly important role. Here, we will show that rewriting mathematical methods by exploiting low-rank tensor decompositions enables the reconstruction of high-dimensional systems which cannot be analyzed by conventional methods. The main contributions of this paper are:
\begin{itemize}[leftmargin=*,itemsep=0ex]
\item combination of data-driven methods with tensor decomposition techniques for the recovery of governing equations, reducing computational costs as well as storage consumption,
\item extension of the pseudoinverse computation in the TT format (developed for \emph{tensor-based dynamic mode decomposition} in~\cite{KLUS2018}),
\item derivation of different basis decompositions in the TT format,
\item demonstration on realistic problems from physics and engineering.
\end{itemize}

This work is organized as follows: In Section \ref{sec: notation}, we introduce SINDy and give a brief overview of the TT format and pseudoinverses in the TT format. In Section \ref{sec: tensor-based SINDy}, we derive the tensor-based reformulation of SINDy. Numerical results are presented in Section \ref{sec: results}. Section \ref{sec: conclusion} concludes with a brief summary and a future outlook.

\section{Notation and preliminaries}\label{sec: notation}

In this section, we will introduce an approach called SINDy to identify ordinary differential equations from data, the tensor-train format, and an algorithm to compute pseudoinverses in the tensor-train format. By combining these techniques, we are able to discover the governing equations of high-dimensional dynamical systems.

\subsection{Sparse identification of nonlinear dynamics}
\label{sec: recovery}

We will consider autonomous ordinary differential equations of the form $\dot{X}(t) = F(X(t))$, where $X(t) \in \mathbb{R}^d$ is the state of the system at time $t$  and $F \colon \mathbb{R}^d \rightarrow \mathbb{R}^d$ is a function. Assuming we have $ m $ measurements of the state of the system, given by $ X_k $, $ k = 1,\dots, m $, and the corresponding time derivatives, given by $ Y_k = \dot{X}_k $, the goal is to reconstruct the function $ F $. To this end, we choose a set of basis functions (also called dictionary) $ \mathbb{D} = \{\psi_1, \psi_2, \dots, \psi_p\} $, with $ \psi_j \colon \mathbb{R}^d \rightarrow \mathbb{R} $, and define the vector-valued function $ \Psi \colon \R^d \to \R^p $ by
\begin{equation*}
    \Psi(X) = \begin{bmatrix} \psi_1(X) & \psi_2(X) & \dots & \psi_p(X) \end{bmatrix}^T.
\end{equation*}
The transformed data matrix $ \Psi(\mathcal{X}) $ is then defined by
\begin{equation}\label{eq: basis matrix}
    \Psi(\mathcal{X}) =
    \begin{bmatrix}
        \Psi(X_1) & \dots & \Psi(X_m)
    \end{bmatrix}
    \in \R^{p \times m}.
\end{equation}
We now wish to determine the coefficient matrix
\begin{equation*}
    \Xi = \begin{bmatrix} \xi_1 & \xi_2 & \dots & \xi_d \end{bmatrix} \in \R^{p \times d}
\end{equation*}
such that the cost function $ \big\lVert \mathcal{Y} - \Xi^T \Psi(\mathcal{X}) \big\rVert_F $ is minimized. Each column vector $ \xi_i $ of $ \Xi $ then corresponds to one function $ F_i $ via
\begin{equation*}
    \left(Y_k\right)_i = F_i(X_k) = \xi_i^T \Psi(X_k),
\end{equation*}
where the entries of $ \xi_i $ determine which basis functions are used for the reconstruction. Thus, we obtain a model of the form $ \dot{X}(t) = \Xi^T \Psi(X(t)) $. One approach to compute an optimal coefficient matrix in the least-squares sense is the use of the pseudoinverse of $\Psi(\mathcal{X})$. That is, we determine $\Xi$ by $ \Xi^T = \mathcal{Y} \cdot (\Psi(\mathcal{X}))^+ $, where $ ^+ $ denotes the pseudoinverse. Additionally, SINDy aims at finding a sparse coefficient matrix $ \Xi $ \cite{BRUNTON2016}. This is accomplished by iteratively removing basis functions corresponding to small entries that might not be required for the reconstruction. Provided that the governing equations can be expressed in terms of the basis functions, SINDy may completely recover the dynamical system. The iterative elimination of basis functions could be omitted if the number of snapshots is large enough so that the solution of the initial least-squares problem is already close to the true solution. Alternatively, a LASSO-based (see, e.g.,~\cite{TIBSHIRANI1996}) approach can be used to identify the coefficients.

If the time derivatives are not available and need to be approximated by finite differences, the resulting $ \dot{X} $ data might be noisy and necessitate denoising techniques. Also measurement data will in general contain noise and require regularization. For more details, see~\cite{BRUNTON2016} and references therein.

\begin{example}[Illustration of SINDy] \label{ex:SINDy examples}
To illustrate the sparse identification process, let us begin with a simple example. Consider Chua's circuit, see, e.g.,~\cite{KILIC2010}, given by
\begin{equation}\label{eq: chua ODE}
    \begin{split}
        \dot{x}_1 &= \alpha (x_2 - x_1 - g(x_1)), \\
        \dot{x}_2 &= x_1 - x_2 + x_3, \\
        \dot{x}_3 &= -\beta x_2,
    \end{split}
\end{equation}
where $ \alpha $ and $ \beta $ are real parameters and $ g \colon \mathbb{R} \to \mathbb{R} $ is a nonlinear function. We will set $ \alpha = 10 $, $ \beta = 14.87 $, and $ g(z) = \delta_1 \, z + \delta_2 \, z \abs{z} $, with $ \delta_1 = -\frac{8}{7} $ and $ \delta_2 = \frac{4}{63} $. Let us try to recover the governing equations of Chua's circuit from data. We simulate the system for $ t = 0, \dots, 20 $ with a step size of $ h = 0.01 $ and the initial condition $X_1 = \left[ -1.13 , 0.004, 0.45\right]^T$, thus $ \mathcal{X}, \mathcal{Y} \in \mathbb{R}^{3 \times 2000} $. Here, we use the exact derivatives.
\begin{enumerate}[label=(\roman*), leftmargin=0.0em, itemindent=1.5em]
\item Using a basis comprising monomials of order up to and including two in each dimension, i.e.,
\begin{equation}\label{chua: first try}
    \Psi_1(X) =
    \begin{bmatrix}
        \; 1 & x_1 & x_1^2 & x_2 & x_1 x_2 & x_1^2 x_2 & \dots & x_2^2 x_3^2 & x_1 x_2^2 x_3^2 & x_1^2 x_2^2 x_3^2 \;
    \end{bmatrix}^T,
\end{equation}
we obtain the following coefficient matrix
\begin{equation*}
    \Xi_1^T = \hspace*{-0.7em}
    \kbordermatrix{
        & 1     & x_1  & x_1^2 & x_2 & x_1 x_2 & x_1^2 x_2 & x_2^2 & x_1 x_2^2 & x_1^2 x_2^2 & x_3 & \dots \\
        & -0.06 & 1.10 & 0.35 &   9.61 & -1.77 &  -1.97 &  1.05 & 3.85 &  1.07 & -0.07 & \dots \\
        &  0    & 1    & 0    &  -1    &  0    &   0    &  0    & 0    &  0    &  1    & \dots \\
        &  0    & 0    & 0    & -14.87 &  0    &   0    &  0    & 0    &  0    &  0    & \dots \\
    }.
\end{equation*}
The coefficients for the second and third function are identified correctly, whereas the first equation cannot be represented by the basis functions and all coefficients are unequal to zero, approximating the missing basis function. The resulting trajectories are shown in Figure~\ref{fig:Chua SINDy} (a).
\item If we use a different set of basis functions, given by
\begin{equation}\label{chua: second try}
    \Psi_2(X) =
    \begin{bmatrix}
        \; 1 & x_1 & x_2 & x_3 & \abs{x_1} & x_1 \abs{x_1} & x_2 \abs{x_1} & x_3 \abs{x_1} & \dots & x_3 \abs{x_3} & x_3 \abs{x_3} \;
    \end{bmatrix}^T
\end{equation}
the system is recovered correctly as
\begin{equation*}
    \Xi_2^T = \hspace*{-0.7em}
    \kbordermatrix{
        & 1 & x_1 & x_2 & x_3 & \abs{x_1} & x_1 \abs{x_1} \\
        & 0 &   1.42 &  10    & 0 & 0 & -0.6349 & \dots \\
        & 0 &   1    & -1     & 1 & 0 &  0      & \dots \\
        & 0 &   0    & -14.87 & 0 & 0 &  0      & \dots
    }.
\end{equation*}
All remaining coefficients are numerically zero. The simulation results are shown in Figure~\ref{fig:Chua SINDy}~(b). \exampleSymbol
\end{enumerate}

\begin{figure}[htb]
    \centering
    \begin{minipage}{0.49\textwidth}
        \centering
        \subfiguretitle{(a)}
        \includegraphics[width=\textwidth]{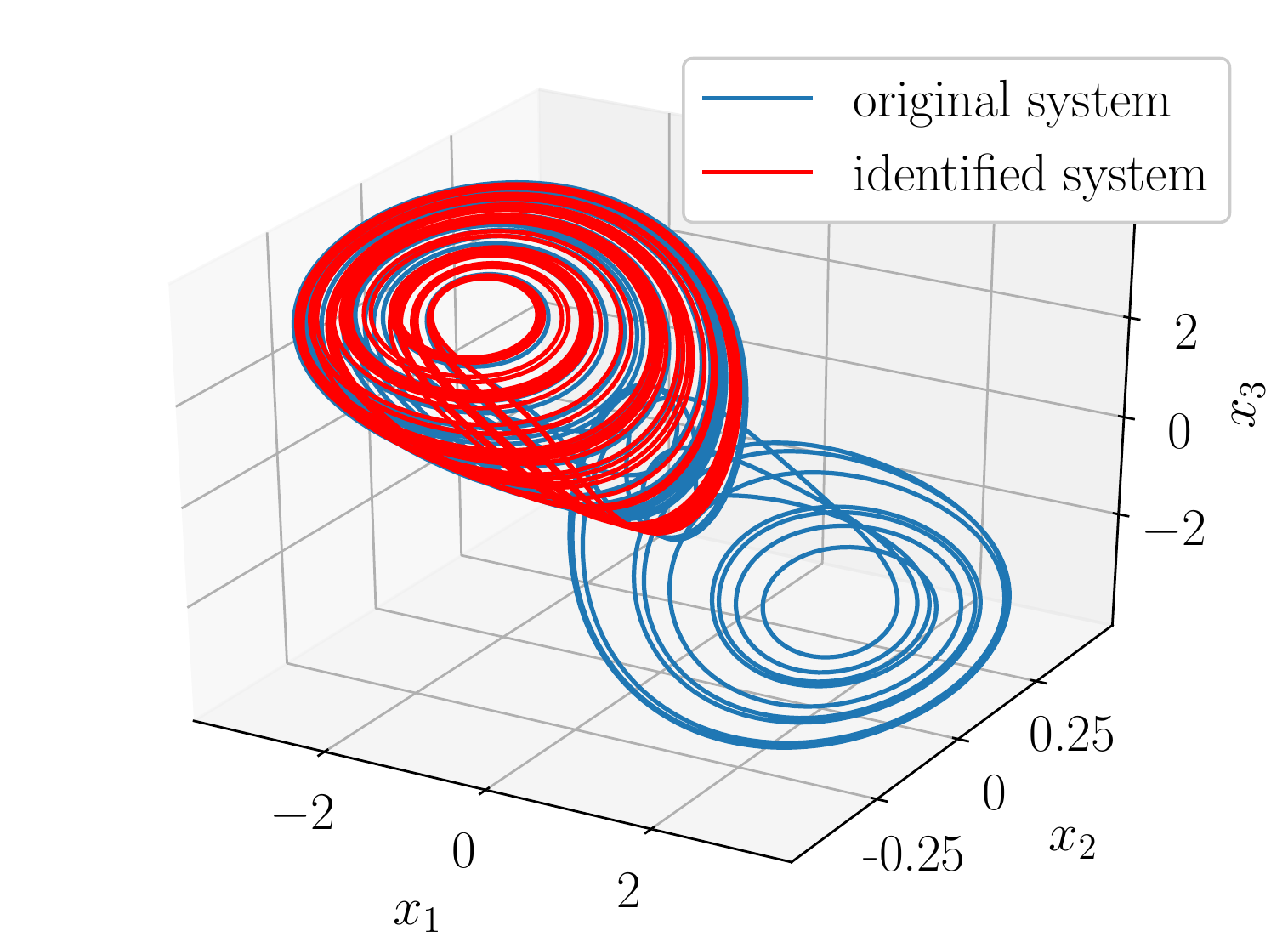}
    \end{minipage}
    \begin{minipage}{0.49\textwidth}
        \centering
        \subfiguretitle{(b)}
        \includegraphics[width=\textwidth]{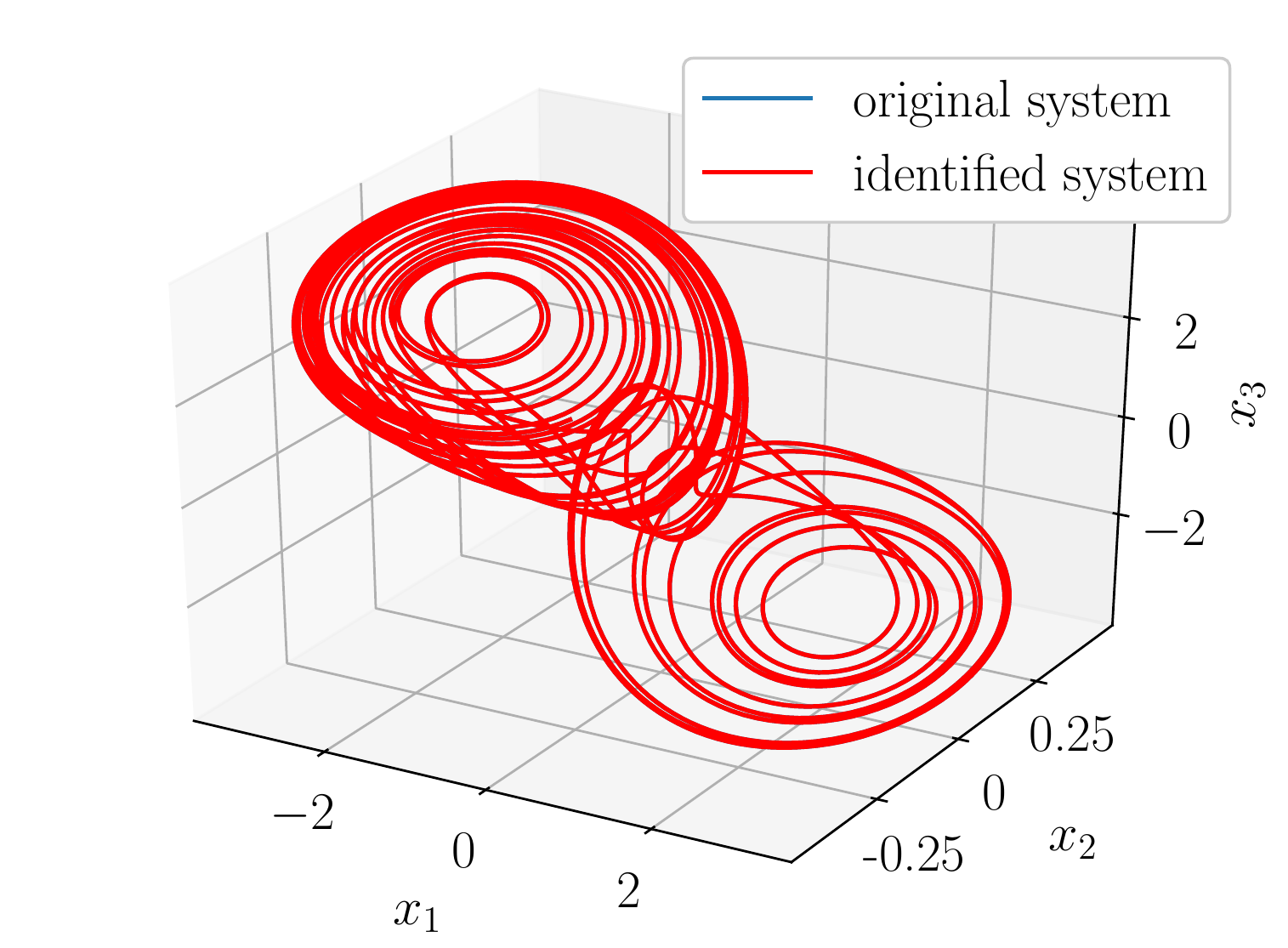}
    \end{minipage}
    \caption{Simulation of Chua's circuit and the identified systems for $ t = 0, \dots, 100 $ with a step size of $ h = 0.01 $. (a) Using a basis comprising only monomials leads to incorrect simulation results, the identified system does not capture the full dynamics. (b) Constructing a basis that contains also $ x_1 \abs{x_1} $ leads to correct results, the original and identified system are identical up to numerical errors.}
    \label{fig:Chua SINDy}
\end{figure}

\end{example}

\begin{remark}[Dependence of SINDy on basis functions]
The example illustrates that in order to obtain an accurate representation of the system, a priori knowledge about the governing equations might be required, which is in general not available when the method is applied to measurement data or data stemming from black-box models. If the basis functions are not chosen appropriately, then SINDy will not faithfully reproduce the dynamics of the system.
\end{remark}

\subsection{The tensor-train format}

Tensors of order $ d $ are multidimensional arrays $\mathbf{T} \in \mathbb{R}^{n_1 \times \dots \times n_d}$. Figure~\ref{fig: tensors} shows a number of simple examples. The different dimensions $n_i \in \mathbb{N}$ for $i=1, \dots , d$ are called \emph{modes}. Elements of tensors are accessed by $\mathbf{T}_{j_1, \dots, j_d} \in \mathbb{R}$, with $1 \leq j_i \leq n_i$. If we fix certain indices, colons are used to indicate the free modes (similar to Matlab's or Python's colon notation). The storage consumption of a tensor can be estimated as $O(n^d)$, where $n$ is the maximum of all mode sizes $ n_1, \dots, n_d $. Storing a $d$-dimensional tensor may be infeasible for large $d$ since the number of elements of a tensor grows exponentially with the order. This is typically called the \emph{curse of dimensionality}. However, it is possible to mitigate this by exploiting low-rank tensor approximations. If the underlying correlation structure admits such a decomposition, an enormous reduction in complexity can be achieved. The basic idea is to decompose tensors $\mathbf{T} \in \mathbb{R}^{n_1 \times \dots \times n_d}$ into different components using tensor products, see~\cite{WILSON1901}.

\begin{figure}[htbp]
\centering
\begin{subfigure}[b]{0.2\textwidth}
\centering
\caption{}
\begin{tikzpicture}
\node (1) [rectangle] at (0,0) {$\begin{bmatrix} 1 \\ 0 \\ 1 \\ 1 \\0 \end{bmatrix}$};
\end{tikzpicture}
\end{subfigure}
\hfill
\begin{subfigure}[b]{0.3\textwidth}
\centering
\caption{}
\begin{tikzpicture}
\node (1) [rectangle] {$\begin{bmatrix} 1 & 0 & 1 & 0\\ 0 & 1 & 1 & 0 \\ 1 & 1 & 0 & 1 \\ 1 & 0 & 1 & 0 \\ 0 & 1 & 0 & 1\end{bmatrix}$};
\end{tikzpicture}
\end{subfigure}
\hfill
\begin{subfigure}[b]{0.4\textwidth}
\centering
\caption{}
\begin{tikzpicture}[
myrect/.style={
  rectangle,
  draw,
  inner sep=0pt,
  fit=#1}
]
\def\da{0.38}
\def\db{0.475}
\node (1) [draw=black, rectangle] at (\da+\da,\db+\db) {$\begin{bmatrix} 1 & 0 & 1 & 1 \\ 1 & 0 & 0 & 0 \\ 0 & 1 & 1 & 0\end{bmatrix}$};
\node (3) [rectangle] at (0,0) {\phantom{\phantom{$\begin{bmatrix} 1 & 1 & 0 & 1 \\ 0 & 1 & 1 & 0 \\ 1 & 0 & 0 & 1\end{bmatrix}$}}};
\draw[-,dotted] (1.south west)--(3.south west);
\node (2) [fill=white, opacity=0.9, rectangle] at (\da,\db) {\phantom{$\begin{bmatrix} 0 & 1 & 1 & 0 \\ 1 & 0 & 0 & 1 \\ 1 & 1 & 1 & 1\end{bmatrix}$}};
\node (2-2) [draw=black, rectangle] at (\da,\db) {$\begin{bmatrix} 0 & 1 & 1 & 0 \\ 1 & 0 & 0 & 1 \\ 1 & 1 & 1 & 1\end{bmatrix}$};
\node (3) [fill=white, opacity=0.9, rectangle] at (0,0) {\phantom{$\begin{bmatrix} 1 & 1 & 0 & 1 \\ 0 & 1 & 1 & 0 \\ 1 & 0 & 0 & 1\end{bmatrix}$}};
\node (3-2) [draw=black, rectangle] at (0,0) {$\begin{bmatrix} 1 & 1 & 0 & 1 \\ 0 & 1 & 1 & 0 \\ 1 & 0 & 0 & 1\end{bmatrix}$};
\draw[-,dotted] (1.north east)--(3.north east);
\draw[-,dotted] (1.north west)--(3.north west);
\draw[-,dotted] (1.south east)--(3.south east);
\end{tikzpicture}
\end{subfigure}
\caption[Low-dimensional tensors represented by arrays]{Low-dimensional tensors represented by arrays: (a) A tensor of order 1 is a vector. (b) A tensor of order 2 is a matrix. (c) A tensor of order 3 can be visualized as layers of matrices.}
\label{fig: tensors}
\end{figure}
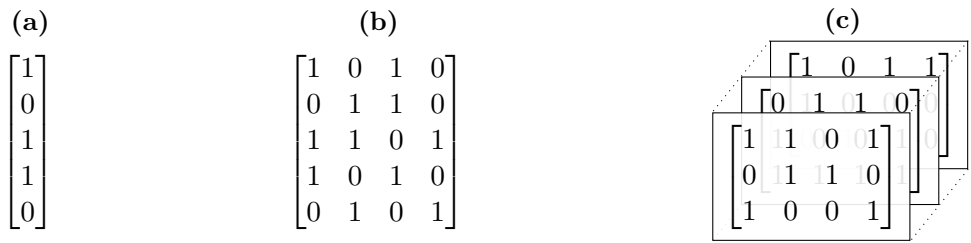

Examples of tensor formats are \emph{rank-one tensors} \cite{FRIEDLAND2013, HACKBUSCH2012} and the \emph{canonical format} \cite{HITCHCOCK1927, KOLDA2009}. A tensor $\mathbf{T} \in \R^{n_1 \times \dots \times n_d}$ of order $d$ is called a rank-one tensor if it can be written as the tensor product of $d$ vectors. A tensor in the canonical format is simply the sum of rank-one tensors. For our purposes, we will rely on the \emph{tensor-train format} (TT format) \cite{OSELEDETS2009b,OSELEDETS2011}, a special case of the \emph{hierarchical Tucker format} \cite{HACKBUSCH2009,ARNOLD2013}. A tensor in the TT format is given by
\begin{equation}
    \mathbf{T} = \sum_{k_0=1}^{r_0} \cdots  \sum_{k_d=1}^{r_d}  \mathbf{T}^{(1)}_{k_0,:,k_1} \otimes \dots \otimes  \mathbf{T}^{(d)}_{k_{d-1},:,k_d}.
\end{equation}
The TT ranks $r_0 , \dots , r_d$, where $r_0 = r_d = 1$, determine the storage consumption of a tensor train and its complexity. Lower ranks imply a lower memory consumption and lower computational costs. Therefore, our aim is to compute low-rank approximations of high-dimensional tensors in the TT format. Figure~\ref{fig: tensor trains} shows the graphical representation -- motivated by~\cite{HOLTZ2012} -- of a tensor train $\mathbf{T} \in \mathbb{R}^{n_1 \times \dots \times n_d}$. A core is depicted by a circle with different arms indicating the modes of the tensor and the rank indices. For more information, we refer to~\cite{GELSS2017b, HACKBUSCH2012}.

\begin{figure}[htbp]
\centering
\begin{tikzpicture}
\draw[black] (0,0) -- node [label={[shift={(0,-0.15)}]$r_1$}] {} ++ (1.5,0) ;
\draw[black] (1.5,0) -- node [label={[shift={(0,-0.15)}]$r_2$}] {} ++ (1.5,0) ;
\draw[black] (3,0) -- ++ (1,0) ;
\draw[black, dotted] (4,0) -- ++ (1,0) ;
\draw[black] (5,0) -- ++ (1,0) ;
\draw[black] (6,0) -- node [label={[shift={(0,-0.19)}]$r_{d-1}$}] {} ++ (1.5,0) ;
\draw[black] (0,0) -- node [label={[shift={(0,-1)}]$n_1$}] {} ++ (0,-0.7) ;
\draw[black] (1.5,0) -- node [label={[shift={(0,-1)}]$n_2$}] {} ++ (0,-0.7) ;
\draw[black] (3,0) -- node [label={[shift={(0,-1)}]$n_3$}] {} ++ (0,-0.7) ;
\draw[black] (6,0) -- node [label={[shift={(0,-1)}]$n_{d-1}$}] {} ++ (0,-0.77) ;
\draw[black] (7.5,0) -- node [label={[shift={(0,-1)}]$n_d$}] {} ++ (0,-0.71) ;
\node[draw,shape=circle,fill=Gray, scale=0.65] at (0,0){};
\node[draw,shape=circle,fill=Gray, scale=0.65] at (1.5,0){};
\node[draw,shape=circle,fill=Gray, scale=0.65] at (3,0){};
\node[draw,shape=circle,fill=Gray, scale=0.65] at (6,0){};
\node[draw,shape=circle,fill=Gray, scale=0.65] at (7.5,0){};
\end{tikzpicture}
\caption{Graphical representation of tensor trains. Here, the first and the last core are considered as matrices, the other cores are tensors of order 3.}
\label{fig: tensor trains}
\end{figure}
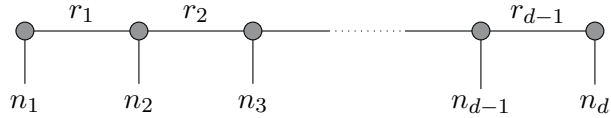

We represent the TT cores as two-dimensional arrays containing vectors as elements. For $\mathbf{T} \in \mathbb{R}^{n_1 \times \dots \times n_d}$ with cores $\mathbf{T}^{(i)} \in \mathbb{R}^{r_{i-1} \times n_i \times r_i}$, a single core is written as
\begin{equation}
    \left\llbracket \mathbf{T}^{(i)} \right\rrbracket = 
    \left\llbracket
    \begin{matrix}
    & \mathbf{T}^{(i)}_{1,:,1} & \cdots & \mathbf{T}^{(i)}_{1,:,r_i} & \\
    & & & & \\
    & \vdots & \ddots & \vdots & \\
    & & & & \\
    & \mathbf{T}^{(i)}_{r_{i-1},:,1} & \cdots & \mathbf{T}^{(i)}_{r_{i-1},:,r_i} &
    \end{matrix}\right\rrbracket.
    \label{eq: core notation - single core}
\end{equation}
We then use the notation $\mathbf{T} = \left\llbracket \mathbf{T}^{(1)}\right\rrbracket \otimes \dots \otimes \left\llbracket \mathbf{T}^{(d)}\right\rrbracket$ for representing tensor trains $\mathbf{T}$, see~\cite{GELSS2017, KAZEEV2012}. This can be regarded as a generalization of the standard matrix multiplication where the cores contain vectors as elements instead of scalar values. Just like multiplying two matrices, we compute the tensor products of the corresponding elements and then sum over the columns and rows, respectively.

In order to construct \emph{matricizations} and \emph{vectorizations} --~also called \emph{tensor unfoldings}~\cite{HOLTZ2012}~-- we define a bijection $\phi_N$ for the mode set $N = (n_1, \dots, n_d)^T \in \mathbb{N}^d$ with
\begin{equation*}
    \begin{gathered}
    \phi_N \colon \{1, \ldots, n_1 \} \times \dots \times \{1, \ldots, n_d \} \rightarrow \{1 , \ldots , \prod_{k=1}^{d} n_k\},\\
    (j_1, \dots, j_d) \mapsto \overline{j_1, \dots, j_d} := \phi_N(j_1, \dots, j_d),
    \end{gathered}
\end{equation*}
where $\overline{j_1, \dots, j_d}$ is the single-index representation of the corresponding multi-index $(j_1, \dots, j_d)$. Typically, bijections based on reverse lexicographic ordering (column-major order in Matlab and Fortran) or colexicographic ordering (row-major order in Numpy and C\texttt{++}) are used, see, e.g.,~\cite{CICHOCKI2016}.

\begin{definition}[Matricization and vectorization] \label{def: matricization}
Let $N = (n_1 , \ldots , n_d)^T$ be a mode set and $\mathbf{T} \in \mathbb{R}^{n_1 \times \dots \times n_d}$ a tensor. For two ordered subsets $N' = (n_1, \dots, n_l)^T$ and $N'' = (n_{l+1} , \ldots , n_d)^T$ of $N$ with $1 \leq l < d$, the \emph{matricization} of $\mathbf{T}$ with respect to $N'$ and $N''$ is given by
\begin{equation}
    \left( \mat{\mathbf{T}}{N'}{N''} \right)_{\overline{j_1 , \ldots , j_l} , \overline{j_{l+1}, \ldots, j_{d}}} =  \mathbf{T}_{j_1, \ldots, j_d}.
    \label{eq_matricization}
\end{equation}
The \emph{vectorization} of $\mathbf{T}$ is given by
\begin{equation*}
\mathrm{vec}(\mathbf{T}) := \left( \mat{\mathbf{T}}{N}{~} \right)_{ \overline{j_1 , \ldots , j_d} } = \mathbf{T}_{j_1 , \ldots , j_d}.
\end{equation*}
\end{definition}
Definition \ref{def: matricization} can also be generalized to arbitrary subsets of the mode set $N$, see, e.g.,~\cite{GELSS2017b}. We will here, however, only need the special case described above. Given a TT core $\mathbf{T}^{(i)} \in \R^{r_{i-1} \times n_i \times r_i}$, we now define
\begin{equation}
    \mathcal{L}\left( \mathbf{T}^{(i)}\right) = \mat{\mathbf{T}^{(i)}}{r_{i-1} , n_i}{r_i} \qquad \textrm{and} \qquad \mathcal{R}\left( \mathbf{T}^{(i)}\right) = \mat{\mathbf{T}^{(i)}}{r_{i-1} }{n_i, r_i}.
    \label{eq: left-/right-unfolding}
\end{equation}

The matricization $\mathcal{L} \left( \mathbf{T}^{(i)}\right)$ is called the \emph{left-unfolding} of $\mathbf{T}^{(i)}$ and $\mathcal{R} \left( \mathbf{T}^{(i)}\right)$ the \emph{right-unfolding} of $\mathbf{T}^{(i)}$, see~\cite{HOLTZ2012}. A TT core is called \emph{left-orthonormal} if its left-unfolding is orthonormal with respect to the columns, i.e.,
\begin{equation*}
 \left( \mathcal{L} \left( \mathbf{T}^{(i)}\right) \right)^T \cdot \mathcal{L} \left( \mathbf{T}^{(i)}\right) = \mat{\mathbf{T}^{(i)}}{r_i}{r_{i-1} , n_i} \cdot \mat{\mathbf{T}^{(i)}}{r_{i-1} , n_i}{r_i} = I \in \R^{r_i \times r_i},
\end{equation*}
and \emph{right-orthonormal} if its right-unfolding is orthonormal with respect to the rows, i.e.,
\begin{equation*}
 \mathcal{R}\left( \mathbf{T}^{(i)}\right)  \cdot \left(  \mathcal{R}\left( \mathbf{T}^{(i)}\right) \right)^T  = \mat{\mathbf{T}^{(i)}}{r_{i-1}}{n_i, r_i}  \cdot \mat{\mathbf{T}^{(i)}}{n_i, r_i}{r_{i-1}}  = I \in \R^{r_{i-1} \times r_{i-1}}.
\end{equation*}
Algorithms for the left- and right-orthonormalization, respectively, can be found in~\cite{OSELEDETS2011, GELSS2017b}. Note that a tensor $\mathbf{T}$ remains the same if we apply an (exact) orthonormalization procedure to it. The algorithms simply compute a different but equivalent representation.

\subsection{Pseudoinverses in the TT format}\label{sec: pseudoinverse}

We now explain how to compute pseudoinverses of matricizations of tensors in the TT format. Later, we aim at applying this technique to the tensorized counterpart of the matrix $\Psi(\mathcal{X})$ defined in \eqref{eq: basis matrix}. It was shown that the pseudoinverse of certain tensor unfoldings can be computed directly in the TT format. We will briefly recapitulate the main idea, details can be found in~\cite{KLUS2018}. Given a tensor $\mathbf{T} \in \mathbb{R}^{n_1 \times \dots \times n_d \times m}$ (e.g., a tensor containing $m$ snapshots of a $d$-dimensional system with dimensions $n_1, \dots , n_d$) in the TT format, we consider the matricization
\begin{equation}
\label{eq: matricization of T}
T = \mat{\mathbf{T}}{n_1 , \dots , n_d}{m}.
\end{equation}
The standard way to calculate the pseudoinverse $T^+$ is based on using its singular value decomposition (SVD), i.e., $ T = U \, \Sigma \, V^T $ with $U^T U = V^T V = I $. The pseudoinverse of $T$ is then given by 
\begin{equation}\label{eq: pseudoinverse in full format}
  T^+ = V \, \Sigma^{-1} U^T. 
\end{equation}

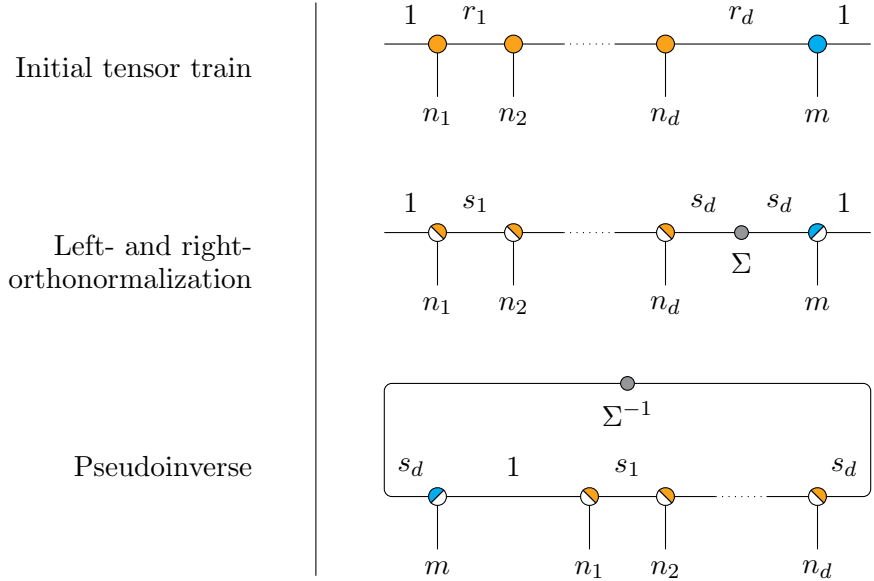
\begin{figure}[htbp]
\centering
\begin{tikzpicture}
\def\g{0}
\draw[black] (-1.6,0.55) --  ++ (0,-7.6) ;
\node[anchor=east] at (-2.3,\g-0.3){Initial tensor train};
\draw[black] (0,\g) -- node [label={[shift={(0,-1)}]$n_1$}] {} ++ (0,-0.7) ;
\draw[black] (1,\g) -- node [label={[shift={(0,-1)}]$n_2$}] {} ++ (0,-0.7) ;
\draw[black] (3,\g) -- node [label={[shift={(0,-1)}]$n_d$}] {} ++ (0,-0.7) ;
\draw[black] (5,\g) -- node [label={[shift={(0,-0.95)}]$m$}] {} ++ (0,-0.7) ;
\draw[black] (0,\g) -- node [label={[shift={(0,0)}]$1$}] {} ++ (-0.7,0) ;
\draw[black] (0,\g) -- node [label={[shift={(0,0)}]$r_1$}] {} ++ (1,0) ;
\draw[black] (1,\g) -- node [] {} ++ (0.66,0) ;
\draw[black, dotted] (1.66,\g) -- node [] {} ++ (0.66,0) ;
\draw[black] (2.33,\g) -- node [] {} ++ (0.66,0) ;
\draw[black] (3,\g) -- node [label={[shift={(0,0)}]$r_d$}] {} ++ (2,0);
\draw[black] (5,\g) -- node [label={[shift={(0,0)}]$1$}] {} ++ (0.7,0) ;
\node[draw,shape=circle,fill=YellowOrange, scale=0.65] at (0,\g){}; 
\node[draw,shape=circle,fill=YellowOrange, scale=0.65] at (1,\g){}; 
\node[draw,shape=circle,fill=YellowOrange, scale=0.65] at (3,\g){}; 
\node[draw,shape=circle,fill=cyan, scale=0.65] at (5,\g){}; 

\def\g{-2.5}
\node[anchor=east] at (-2.3,\g-0.2){Left- and right-}; 
\node[anchor=east] at (-2.3,\g-0.6){orthonormalization}; 
\draw[black] (0,\g) -- node [label={[shift={(0,-1)}]$n_1$}] {} ++ (0,-0.7) ;
\draw[black] (1,\g) -- node [label={[shift={(0,-1)}]$n_2$}] {} ++ (0,-0.7) ;
\draw[black] (3,\g) -- node [label={[shift={(0,-1)}]$n_d$}] {} ++ (0,-0.7) ;
\draw[black] (5,\g) -- node [label={[shift={(0,-0.95)}]$m$}] {} ++ (0,-0.7) ;
\draw[black] (0,\g) -- node [label={[shift={(0,0)}]$1$}] {} ++ (-0.7,0) ;
\draw[black] (0,\g) -- node [label={[shift={(0,0)}]$s_1$}] {} ++ (1,0) ;
\draw[black] (1,\g) -- node [] {} ++ (0.66,0) ;
\draw[black, dotted] (1.66,\g) -- node [] {} ++ (0.66,0) ;
\draw[black] (2.33,\g) -- node [] {} ++ (0.66,0) ;
\draw[black] (3,\g) -- node [label={[shift={(0,0)}]$s_d$}] {} ++ (1,0);
\draw[black] (4,\g) -- node [label={[shift={(0,0)}]$s_d$}] {} ++ (1,0);
\draw[black] (5,\g) -- node [label={[shift={(0,0)}]$1$}] {} ++ (0.7,0) ;
\node[draw,shape=semicircle,rotate=135 ,fill=white, anchor=south,inner sep=2pt, outer sep=0pt, scale=0.75] at (0,\g){}; 
\node[draw,shape=semicircle,rotate=315 ,fill=YellowOrange, anchor=south,inner sep=2pt, outer sep=0pt, scale=0.75] at (0,\g){};
\node[draw,shape=semicircle,rotate=135 ,fill=white, anchor=south,inner sep=2pt, outer sep=0pt, scale=0.75] at (1,\g){}; 
\node[draw,shape=semicircle,rotate=315 ,fill=YellowOrange, anchor=south,inner sep=2pt, outer sep=0pt, scale=0.75] at (1,\g){};
\node[draw,shape=semicircle,rotate=135 ,fill=white, anchor=south,inner sep=2pt, outer sep=0pt, scale=0.75] at (3,\g){}; 
\node[draw,shape=semicircle,rotate=315 ,fill=YellowOrange, anchor=south,inner sep=2pt, outer sep=0pt, scale=0.75] at (3,\g){};
\node[draw,shape=semicircle,rotate=45  ,fill=cyan, anchor=south,inner sep=2pt, outer sep=0pt, scale=0.75] at (5,\g){}; 
\node[draw,shape=semicircle,rotate=225 ,fill=white, anchor=south,inner sep=2pt, outer sep=0pt, scale=0.75] at (5,\g){};
\node[draw,shape=circle,fill=Gray, scale=0.5] at (4,\g){}; 
\node[] at (4,\g-0.4){$\Sigma$}; 

\def\g{-6}
\node[anchor=east] at (-2.3,\g+0.4){Pseudoinverse}; 
\draw[black] (2,\g) -- node [label={[shift={(0,-1)}]$n_1$}] {} ++ (0,-0.7) ;
\draw[black] (3,\g) -- node [label={[shift={(0,-1)}]$n_2$}] {} ++ (0,-0.7) ;
\draw[black] (5,\g) -- node [label={[shift={(0,-1)}]$n_d$}] {} ++ (0,-0.7) ;
\draw[black] (0,\g) -- node [label={[shift={(0,-0.95)}]$m$}] {} ++ (0,-0.7) ;
\draw[black] (0,\g) -- node [label={[shift={(0,0)}]$1$}] {} ++ (2,0);
\draw[black] (2,\g) -- node [label={[shift={(0,0)}]$s_1$}] {} ++ (1,0);
\draw[black] (3,\g) -- node [] {} ++ (0.66,0) ;
\draw[black, dotted] (3.66,\g) -- node [] {} ++ (0.66,0) ;
\draw[black] (4.33,\g) -- node [] {} ++ (0.66,0) ;
\draw[black, rounded corners=0.1cm] (0,\g) -- node [label={[shift={(0,0)}]$s_d$}] {} ++ (-0.7,0) -- ++ (0,1.5) -- ++ (6.4,0) -- ++ (0,-1.5) -- node [label={[shift={(0,0)}]$s_d$}] {} ++ (-0.7,0); 
\node[draw,shape=semicircle,rotate=135 ,fill=white, anchor=south,inner sep=2pt, outer sep=0pt, scale=0.75] at (2,\g){}; 
\node[draw,shape=semicircle,rotate=315 ,fill=YellowOrange, anchor=south,inner sep=2pt, outer sep=0pt, scale=0.75] at (2,\g){};
\node[draw,shape=semicircle,rotate=135 ,fill=white, anchor=south,inner sep=2pt, outer sep=0pt, scale=0.75] at (3,\g){}; 
\node[draw,shape=semicircle,rotate=315 ,fill=YellowOrange, anchor=south,inner sep=2pt, outer sep=0pt, scale=0.75] at (3,\g){};
\node[draw,shape=semicircle,rotate=135 ,fill=white, anchor=south,inner sep=2pt, outer sep=0pt, scale=0.75] at (5,\g){}; 
\node[draw,shape=semicircle,rotate=315 ,fill=YellowOrange, anchor=south,inner sep=2pt, outer sep=0pt, scale=0.75] at (5,\g){};
\node[draw,shape=semicircle,rotate=45  ,fill=cyan, anchor=south,inner sep=2pt, outer sep=0pt, scale=0.75] at (0,\g){}; 
\node[draw,shape=semicircle,rotate=225 ,fill=white, anchor=south,inner sep=2pt, outer sep=0pt, scale=0.75] at (0,\g){};
\node[draw,shape=circle,fill=Gray, scale=0.5] at (2.5,\g+1.5){}; 
\node[] at (2.5,\g+1.5-0.4){$\Sigma^{-1}$}; 
\end{tikzpicture}
\caption[]{Computation of the pseudoinverse of a tensor train: After left- and right-orthonormalization of the initial tensor train (half-filled circles), the pseudoinverse can be represented as a cyclic tensor train with reordered cores. Note that we here depict the special case of the pseudoinversion of a tensor $\mathbf{T} \in \mathbb{R}^{n_1 \times \dots \times n_d \times m}$, the general case is considered in~\cite{KLUS2018}.}
\label{fig: Pseudoinverse}
\end{figure}

In order to directly construct the pseudoinverse $T^+$ from a TT representation of $\mathbf{T}$, we apply the two orthonormalization procedures. That is, we left-orthonormalize the TT cores from $\mathbf{T}^{(1)}$ to $\mathbf{T}^{(d)}$ and right-orthonormalize the core $\mathbf{T}^{(d+1)}$. Doing this, we obtain a global SVD of the matricization $T$. Similar to \eqref{eq: pseudoinverse in full format}, the pseudoinverse $T^+$ can then be obtained by reordering the cores, see Figure~\ref{fig: Pseudoinverse}. A detailed description of the corresponding algorithms can be found in~\cite{KLUS2018}, where we have shown how to generalize the orthonormalization procedure to arbitrary matricizations of tensor trains. Algorithm~\ref{alg:Pseudoinverse} describes the pseudoinversion procedure for a tensor train $\mathbf{T} \in \mathbb{R}^{n_1 \times \dots \times n_d \times m}$ in detail.

\begin{algorithm}[htbp]
    \caption{Pseudoinversion of tensor trains.}
    \label{alg:Pseudoinverse}
    \begin{algorithmic}[1]
    \Require \parbox[t]{\dimexpr\linewidth-31.5pt}{Tensor train $\mathbf{T} \in \mathbb{R}^{n_1 \times \dots \times n_d \times m}$.\vphantom{$T = \mat{\mathbf{T}}{n_1, \dotsc, n_d}{ m}$}}
    \Ensure Pseudoinverse of $T = \mat{\mathbf{T}}{n_1, \dotsc, n_d}{ m}$.
    \algrule
    \State Left-orthonormalize $\mathbf{T}^{(1)}, \dotsc, \mathbf{T}^{(d-1)}$ and right-orthonormalize $\mathbf{T}^{(d+1)}$.
    \State Compute SVD of $\mathcal{L} \left( \mathbf{T}^{(d)} \right)$, i.e., $\mathcal{L} \left( \mathbf{T}^{(d)} \right) = U \Sigma V^T$ with $\Sigma \in \mathbb{R}^{s \times s}$.
    \State Define $\mathbf{U} \in \mathbb{R}^{r_{d-1} \times n_d \times s}$ as a reshaped version of $U$ with $\mathbf{U}_{k,x,l } = U_{\overline{k, x}, l}$.
    \State Define $\mathbf{V} \in \mathbb{R}^{s \times m }$ by $\mathcal{R} \left( \mathbf{V} \right) = V^T \cdot \mathcal{R} \left( \mathbf{T}^{(d+1)} \right)$.
    \State Set $\mathbf{T}^{(d)}$  to $\mathbf{U}$, $\mathbf{T}^{(d+1)}$  to $\mathbf{V}$, and $r_d$ to $s$.
    \State Define $\tilde{U} = \mat{\left( \sum_{k_0 =1}^{r_0} \cdots \sum_{k_{d-1}=1}^{r_{d-1}} \mathbf{T}^{(1)}_{k_0, :, k_1 } \otimes \dots \otimes \mathbf{T}^{(d)}_{k_{d-1}, :, : } \right)}{n_1, \dotsc, n_d}{r_d}$.
    \State Define $\tilde{V} = \mat{\mathbf{T}^{(d+1)}}{m}{r_d}$.
    \State Define $ T^+ = \tilde{V} \, \Sigma^{-1} \, \tilde{U}^T$.
    \end{algorithmic}
\end{algorithm}

An important aspect is that we do not need to compute the pseudoinverse of $T$ explicitly. Instead, we only orthonormalize the TT cores and compute the matrix $\Sigma$ by executing the lines 1 to 5 of Algorithm~\ref{alg:Pseudoinverse}. We then store the representation
\begin{equation*}
    \begin{split}
        \mathbf{T}^+ = \sum_{k_1 = 1}^{r_1} \dots \sum_{k_{d} = 1}^{r_{d}} \sigma^{-1}_{ k_d} \cdot   \mathbf{T}^{(d+1)}_{k_{d}, :, 1 } \otimes  \mathbf{T}^{(1)}_{1, :, k_{1} } \otimes \dots \otimes \mathbf{T}^{(d)}_{k_{d-1}, :, k_d } \in \mathbb{R}^{m \times n_1 \times \dots \times n_d},
    \end{split}
\end{equation*}
which can be either regarded as the sum of $r_d$ tensor trains scaled by $\sigma^{-1}_{1}, \dotsc, \sigma^{-1}_{r_d} $ or as a cyclic tensor train \cite{HACKBUSCH2012} as depicted in Figure~\ref{fig: Pseudoinverse}. It holds that
\begin{equation*}
 \mat{\mathbf{T}^+}{m}{n_1, \dots, n_d} = \left( \mat{\mathbf{T}}{n_1, \dots, n_d}{m}\right)^+.
\end{equation*}

For the orthonormalization procedures, we only consider compact/reduced SVDs, i.e., only the nonzero singular values and corresponding singular vectors are stored. It is also possible to use truncated SVDs within Algorithm \ref{alg:Pseudoinverse}, i.e., we discard all singular values $\sigma_k$ with $\sigma_k / \sigma_\textrm{max} < \varepsilon$, where $\sigma_\textrm{max}$ is the largest singular value and $\varepsilon$ a given threshold. In this way, we can reduce the computational costs and the storage consumption even further.

\begin{lemma}[Complexity of the pseudoinverse computation]\label{lemma: pseudoinverse - effort}
  The computational effort of Algorithm \ref{alg:Pseudoinverse} can be estimated as $O(d \cdot n \cdot r^3 + m \cdot r^2)$, where 
  $n$ is the maximum of the first $d$ modes and $r$ the maximum of all TT ranks. 
\end{lemma} 
\begin{proof}
  The overall computational costs of Algorithm \ref{alg:Pseudoinverse} can be estimated by the number of applied SVDs. The complexity of calculating an SVD of a left-/right-unfolding can be estimated as $O(n \cdot r^3)$. Since we compute the left-orthonormalizations of the first $d$ cores and the right-orthonormalization of $\mathcal{R}(\mathbf{T}^{(d+1)}) \in \mathbb{R}^{r \times m}$, we obtain the estimation $O(d \cdot n \cdot r^3 + m \cdot r^2)$, assuming that $r \leq m$.
\end{proof}

\section{Tensor-based reformulation of SINDy}\label{sec: tensor-based SINDy}

After introducing SINDy and the TT format, we will now show how to combine the data-driven recovery of dynamical systems with tensor decompositions. We restrict ourselves to a set of basis functions $ \mathbb{D} = \{\psi_1 , \dots , \psi_p\}$, $\psi_j \colon \mathbb{R} \rightarrow \mathbb{R}$, and corresponding basis matrices whose entries are given by products of the given basis functions applied to the different dimensions of all snapshots. Note that the basis functions $\psi_j$, $j = 1, \dots , p$, are now defined on $\mathbb{R}$ and not on $\mathbb{R}^d$ as in Section~\ref{sec: recovery}. Exploiting the TT format for the construction of the tensorized counterpart of the basis matrix $\Psi(\mathcal{X})$, we are able to express large numbers of combinations of the basis functions $\psi_1, \dots , \psi_p$ in a compact way as tensor-products of the one-dimensional basis functions. By doing this, we can not only reduce the storage consumption of the basis matrix but may also lower the computational costs compared to the conventional SINDy-like implementation. 

In what follows, we will describe different approaches for the tensor-based construction and show how to solve the minimization problem 
\begin{equation}\label{eq: min problem in TT}
 \min_{\mathbf{\Xi} } \lVert \mathcal{Y} - \mathbf{\Xi}^T \mathbf{\Psi}(\mathcal{X}) \rVert_F
\end{equation}
directly in the TT format. Here, $\mathbf{\Xi}$ and $\mathbf{\Psi}(\mathcal{X})$ denote the tensorized counterparts of the matrices $\Xi$ and $\Psi(\mathcal{X})$ introduced in Section~\ref{sec: notation}. As mentioned in the introduction, we call our method MANDy.

\subsection{Basis decomposition}

We will consider two different types of basis decompositions. Let $X = [x_1, \dots, x_d]^T \in \mathbb{R}^d$ be a vector and $\psi_j : \mathbb{R} \rightarrow \mathbb{R}$, $j = 1, \dots, p$, basis functions. We consider the two rank-one decompositions
\begin{equation}
\label{eq: psi f,x}
\mathbf{\Psi}_{\textrm{cm}}(X) = \mathbf{\Psi}_{\textrm{cm}}^{(1)}(X) \otimes \dots \otimes \mathbf{\Psi}_{\textrm{cm}}^{(d)}(X) = \begin{bmatrix} \psi_1 (x_1) \\ \vdots \\ \psi_p(x_1) \end{bmatrix} \otimes \dots \otimes \begin{bmatrix} \psi_1 (x_d) \\ \vdots \\ \psi_p(x_d) \end{bmatrix} \in \mathbb{R}^{p \times \dots \times p}
\end{equation}
and
\begin{equation}
\label{eq: psi x,f}
\mathbf{\Psi}_{\textrm{fm}}(X) = \mathbf{\Psi}_{\textrm{fm}}^{(1)}(X) \otimes \dots \otimes \mathbf{\Psi}_{\textrm{fm}}^{(p)}(X) = \begin{bmatrix} \psi_1 (x_1) \\ \vdots \\ \psi_1(x_d) \end{bmatrix} \otimes \dots \otimes \begin{bmatrix} \psi_p (x_1) \\ \vdots \\ \psi_p(x_d) \end{bmatrix} \in \mathbb{R}^{d \times \dots \times d}.
\end{equation}
That is, for \eqref{eq: psi f,x} we apply all basis functions to one specific element of the vector $X$ in each core while we apply only one basis function to all elements of $X$ in each core of \eqref{eq: psi x,f}. Analogously to the row and column major order of multidimensional arrays, we call $\mathbf{\Psi}_\textrm{cm}$ a \emph{coordinate-major decomposition} and $\mathbf{\Psi}_\textrm{fm}$ a \emph{function-major decomposition}. It holds that 
\begin{equation*}
    \left( \mathbf{\Psi}_\textrm{cm}(X) \right)_{j_1 , \dots , j_d} = \psi_{j_1}(x_1) \cdot \ldots \cdot  \psi_{j_d}(x_d) \quad \text{and} \quad \left( \mathbf{\Psi}_\textrm{fm}(X) \right)_{j_1, \dots, j_p} = \psi_1(x_{j_1}) \cdot \ldots \cdot \psi_p(x_{j_p}). 
\end{equation*}
Using the decompositions \eqref{eq: psi f,x} and \eqref{eq: psi x,f}, we can express the tensors $\mathbf{\Psi}_\textrm{cm}(X)$ and $\mathbf{\Psi}_\textrm{fm}(X)$ in a storage-efficient way. The memory consumption of the full representations of the tensors is $O(p^d)$ and $O(d^p)$, respectively, whereas the memory consumption of both rank-one representations is $O(p \cdot d)$.

\begin{remark}[Choice of rank-one decompositions]
What we present here are only two types of rank-one decompositions for a set of basis functions. In fact, other decompositions are also possible and specific problems might necessitate more complex representations in the future. However, we will focus on decompositions of the form \eqref{eq: psi f,x} or \eqref{eq: psi x,f}.
\end{remark}

\subsection{Multidimensional approximation of nonlinear dynamical systems}

For $m$ different vectors $X_k = \left[x_{k,1}, \dots , x_{k,d}\right]^T \in \mathbb{R}^d$, $k = 1, \dots , m$, stored in a matrix $\mathcal{X}$ (see Section \ref{sec: recovery}), we aim at using the decompositions \eqref{eq: psi f,x} and \eqref{eq: psi x,f} to construct counterparts of the matrices
\begin{equation}\label{eq: matrix coordinate major}
    \Psi_\textrm{cm}(\mathcal{X}) =
    \begin{bmatrix}
        \Psi_\textrm{cm}(X_1) & \Psi_\textrm{cm}(X_2) & \dots & \Psi_\textrm{cm}(X_m)
    \end{bmatrix} \in \mathbb{R}^{p^d \times m}
\end{equation}
and
\begin{equation}\label{eq: matrix function major}
    \Psi_\textrm{fm}(\mathcal{X}) =
    \begin{bmatrix}
        \Psi_\textrm{fm}(X_1) & \Psi_\textrm{fm}(X_2) & \dots & \Psi_\textrm{fm}(X_m)
    \end{bmatrix} \in \mathbb{R}^{d^p \times m}
\end{equation}
directly in the TT format. Here, $\Psi_\textrm{cm}(X_k) \in \mathbb{R}^{p^d}$ and $\Psi_\textrm{fm}(X_k) \in \mathbb{R}^{d^p}$ denote the vectorizations $\textrm{vec}\left(\mathbf{\Psi}_\textrm{cm}(X_k)\right)$ and $\textrm{vec}\left(\mathbf{\Psi}_\textrm{fm}(X_k)\right)$, respectively. Collecting the rank-one decompositions given in \eqref{eq: psi f,x} for all vectors $X_1, \dots, X_m$ in a single TT decomposition, we obtain
\begin{equation}\label{eq: TT coordinate major}
\begin{split}
 \mathbf{\Psi}_\textrm{cm}(\mathcal{X}) &= \left \llbracket \mathbf{\Psi}_\textrm{cm} ^{(1)} (\mathcal{X})\right \rrbracket \otimes \left \llbracket \mathbf{\Psi}_\textrm{cm} ^{(2)} (\mathcal{X})\right \rrbracket \otimes \dots \otimes \left \llbracket \mathbf{\Psi}_\textrm{cm} ^{(d)} (\mathcal{X})\right \rrbracket \otimes \left \llbracket \mathbf{\Psi}_\textrm{cm} ^{(d+1)} (\mathcal{X})\right \rrbracket \\
 &= \left \llbracket \begin{matrix}
  \mathbf{\Psi}_\textrm{cm}^{(1)} (X_1) & \cdots & \mathbf{\Psi}_\textrm{cm}^{(1)} (X_m) 
 \end{matrix} \right\rrbracket \otimes
 \left\llbracket\begin{matrix}
  \mathbf{\Psi}_\textrm{cm}^{(2)} (X_1) & & 0 \\
   & \ddots & \\
   0 & & \mathbf{\Psi}_\textrm{cm}^{(2)} (X_m)
 \end{matrix} \right\rrbracket \otimes \cdots \\
 & \qquad \cdots \otimes
 \left \llbracket \begin{matrix}
  \mathbf{\Psi}_\textrm{cm}^{(d)} (X_1) & & 0 \\
   & \ddots & \\
   0 & & \mathbf{\Psi}_\textrm{cm}^{(d)} (X_m) 
 \end{matrix} \right\rrbracket \otimes
 \left \llbracket \begin{matrix}
  e_1 \\
  \vdots \\
  e_m
 \end{matrix} \right\rrbracket \in \mathbb{R}^{p \times  \dots \times p \times m},
\end{split}
\end{equation}
where $e_k$, $k = 1, \dots , m$, denote the unit vectors of the standard basis in the $m$-dimensional Euclidean space. Analogously, we can write
\begin{equation}\label{eq: TT function major}
\begin{split}
 \mathbf{\Psi}_\textrm{fm}(\mathcal{X}) &= \left \llbracket \mathbf{\Psi}_\textrm{fm} ^{(1)} (\mathcal{X})\right \rrbracket \otimes \left \llbracket \mathbf{\Psi}_\textrm{fm} ^{(2)} (\mathcal{X})\right \rrbracket \otimes \dots \otimes \left \llbracket \mathbf{\Psi}_\textrm{fm} ^{(p)} (\mathcal{X})\right \rrbracket \otimes \left \llbracket \mathbf{\Psi}_\textrm{fm} ^{(p+1)} (\mathcal{X})\right \rrbracket \\
 &= 
 \left \llbracket \begin{matrix}
  \mathbf{\Psi}_\textrm{fm}^{(1)} (X_1) & \cdots & \mathbf{\Psi}_\textrm{fm}^{(1)} (X_m) 
 \end{matrix} \right\rrbracket \otimes
 \left \llbracket \begin{matrix}
  \mathbf{\Psi}_\textrm{fm}^{(2)} (X_1) & & 0 \\
   & \ddots & \\
   0 & & \mathbf{\Psi}_\textrm{fm}^{(2)} (X_m) 
 \end{matrix} \right\rrbracket \otimes \cdots \\
 & \qquad \cdots \otimes
 \left \llbracket \begin{matrix}
  \mathbf{\Psi}_\textrm{fm}^{(p)} (X_1) & & 0 \\
   & \ddots & \\
   0 & & \mathbf{\Psi}_\textrm{fm}^{(p)} (X_m) 
 \end{matrix} \right\rrbracket \otimes
 \left \llbracket \begin{matrix}
  e_1 \\
  \vdots \\
  e_m
 \end{matrix} \right\rrbracket \in \mathbb{R}^{d \times  \dots \times d \times m}.
\end{split}
\end{equation}
Both TT decompositions \eqref{eq: TT coordinate major} and \eqref{eq: TT function major} have a TT rank of $m$ and it holds that
\begin{equation*}
 \mat{\mathbf{\Psi}_\textrm{cm}(\mathcal{X})}{p , \dots, p}{m} = \Psi_\textrm{cm}(\mathcal{X}) \quad \textrm{and} \quad \mat{\mathbf{\Psi}_\textrm{fm}(\mathcal{X})}{d , \dots, d}{m} = \Psi_\textrm{fm}(\mathcal{X}).
\end{equation*}
We could also express the basis tensors in the canonical format, but using the TT format enables us to construct the pseudoinverse of $\mathbf{\Psi}_\textrm{cm}(\mathcal{X})$ and $\mathbf{\Psi}_\textrm{fm}(\mathcal{X})$ directly as a TT decomposition, see Section \ref{sec: pseudoinverse}. That is, we solve the minimization problem \eqref{eq: min problem in TT} by computing the pseudoinverse of $\mathbf{\Psi}_\textrm{cm}(\mathcal{X})$ or $\mathbf{\Psi}_\textrm{fm}(\mathcal{X})$, respectively, in the TT format and obtain 
\begin{equation}\label{eq: solution in TT}
  \mathbf{\Xi}^T = \mathcal{Y} \cdot \mathbf{\Psi}_{\textrm{cm}/\textrm{fm}}(\mathcal{X})^+
\end{equation}
with
\begin{equation*}
 \mat{\mathbf{\Xi}^T}{d}{p/d, \dots , p/d} = \mathcal{Y} \cdot \mat{\mathbf{\Psi}_{\textrm{cm}/\textrm{fm}}(\mathcal{X})^+}{m}{p/d, \dots, p/d} = \mathcal{Y} \cdot \Psi_{\textrm{cm}/\textrm{fm}}(\mathcal{X})^+.
\end{equation*}
For a detailed description of the tensor contraction in \eqref{eq: solution in TT}, see Figure~\ref{fig: contraction}.

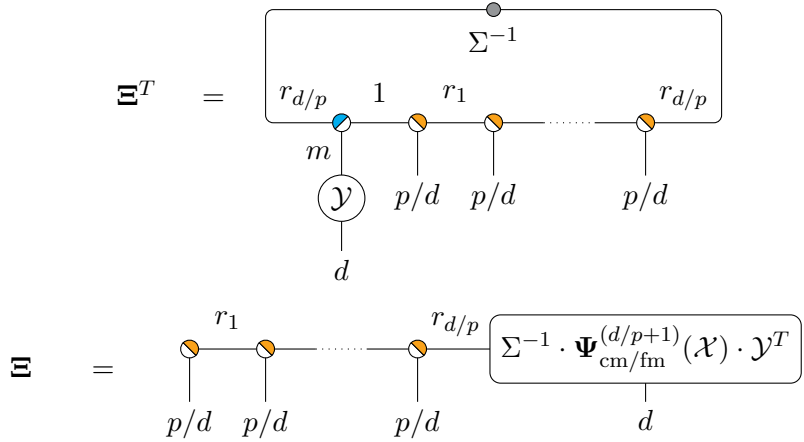
\begin{figure}[htbp]
\centering
\begin{tikzpicture}
\def\g{0}
\node[anchor=east] at (-2.3,\g+0.4){$\mathbf{\Xi}^T$}; 
\node[anchor=east] at (-1.4,\g+0.3){$=$}; 
\draw[black] (1,\g) -- node [label={[shift={(0,-1.1)}]$p/d$}] {} ++ (0,-0.7) ;
\draw[black] (2,\g) -- node [label={[shift={(0,-1.1)}]$p/d$}] {} ++ (0,-0.7) ;
\draw[black] (4,\g) -- node [label={[shift={(0,-1.1)}]$p/d$}] {} ++ (0,-0.7) ;
\draw[black] (0,\g) -- node [label={[shift={(-0.3,-0.25)}]$m$}] {} ++ (0,-1) ; 
\draw[black] (0,\g-1) -- node [label={[shift={(0,-1)}]$d$}] {} ++ (0,-0.7) ;
\node[draw,shape=circle, fill=white, inner sep = 2pt] at (0,\g-1){$\mathcal{Y}$};
\draw[black] (0,\g) -- node [label={[shift={(0,0)}]$1$}] {} ++ (1,0);
\draw[black] (1,\g) -- node [label={[shift={(0,0)}]$r_1$}] {} ++ (1,0);
\draw[black] (2,\g) -- node [] {} ++ (0.66,0) ;
\draw[black, dotted] (2.66,\g) -- node [] {} ++ (0.66,0) ;
\draw[black] (3.33,\g) -- node [] {} ++ (0.66,0) ;
\draw[black, rounded corners=0.1cm] (0,\g) -- node [label={[shift={(0,-0.1)}]$r_{d/p}$}] {} ++ (-1,0) -- ++ (0,1.5) -- ++ (6,0) -- ++ (0,-1.5) -- node [label={[shift={(0,-0.1)}]$r_{d/p}$}] {} ++ (-1,0); 
\node[draw,shape=semicircle,rotate=135 ,fill=white, anchor=south,inner sep=2pt, outer sep=0pt, scale=0.75] at (1,\g){}; 
\node[draw,shape=semicircle,rotate=315 ,fill=YellowOrange, anchor=south,inner sep=2pt, outer sep=0pt, scale=0.75] at (1,\g){};
\node[draw,shape=semicircle,rotate=135 ,fill=white, anchor=south,inner sep=2pt, outer sep=0pt, scale=0.75] at (2,\g){}; 
\node[draw,shape=semicircle,rotate=315 ,fill=YellowOrange, anchor=south,inner sep=2pt, outer sep=0pt, scale=0.75] at (2,\g){};
\node[draw,shape=semicircle,rotate=135 ,fill=white, anchor=south,inner sep=2pt, outer sep=0pt, scale=0.75] at (4,\g){}; 
\node[draw,shape=semicircle,rotate=315 ,fill=YellowOrange, anchor=south,inner sep=2pt, outer sep=0pt, scale=0.75] at (4,\g){};
\node[draw,shape=semicircle,rotate=45  ,fill=cyan, anchor=south,inner sep=2pt, outer sep=0pt, scale=0.75] at (0,\g){}; 
\node[draw,shape=semicircle,rotate=225 ,fill=white, anchor=south,inner sep=2pt, outer sep=0pt, scale=0.75] at (0,\g){};
\node[draw,shape=circle,fill=Gray, scale=0.5] at (2,\g+1.5){}; 
\node[] at (2,\g+1.5-0.4){$\Sigma^{-1}$}; 

\def\g{-3}
\def\f{-2}
\node[anchor=east] at (-1.7+\f,\g-0.2){$\mathbf{\Xi}^{\phantom{T}}$}; 
\node[anchor=east] at (-0.85+\f,\g-0.3){$=$}; 
\draw[black] (0+\f,\g) -- node [label={[shift={(0,-1.1)}]$p/d$}] {} ++ (0,-0.7) ;
\draw[black] (1+\f,\g) -- node [label={[shift={(0,-1.1)}]$p/d$}] {} ++ (0,-0.7) ;
\draw[black] (3+\f,\g) -- node [label={[shift={(0,-1.1)}]$p/d$}] {} ++ (0,-0.7) ;
\draw[black] (6+\f,\g) -- node [label={[shift={(0,-1)}]$d$}] {} ++ (0,-0.7) ;
\draw[black] (0+\f,\g) -- node [label={[shift={(0,0)}]$r_1$}] {} ++ (1,0);
\draw[black] (3+\f,\g) -- node [label={[shift={(0,-0.1)}]$r_{d/p}$}] {} ++ (1,0);
\draw[black] (1+\f,\g) -- node [] {} ++ (0.66,0) ;
\draw[black, dotted] (1.66+\f,\g) -- node [] {} ++ (0.66,0) ;
\draw[black] (2.33+\f,\g) -- node [] {} ++ (0.66,0) ;
\node[draw,rounded corners, fill=white, inner sep = 4pt] at (6+\f,\g){\raisebox{-1em}{$\Sigma^{-1} \cdot \mathbf{\Psi}_{\textrm{cm}/\textrm{fm}}^{(d/p +1)}(\mathcal{X}) \cdot \mathcal{Y}^T$}};
\node[draw,shape=semicircle,rotate=135 ,fill=white, anchor=south,inner sep=2pt, outer sep=0pt, scale=0.75] at (0+\f,\g){}; 
\node[draw,shape=semicircle,rotate=315 ,fill=YellowOrange, anchor=south,inner sep=2pt, outer sep=0pt, scale=0.75] at (0+\f,\g){};
\node[draw,shape=semicircle,rotate=135 ,fill=white, anchor=south,inner sep=2pt, outer sep=0pt, scale=0.75] at (1+\f,\g){}; 
\node[draw,shape=semicircle,rotate=315 ,fill=YellowOrange, anchor=south,inner sep=2pt, outer sep=0pt, scale=0.75] at (1+\f,\g){};
\node[draw,shape=semicircle,rotate=135 ,fill=white, anchor=south,inner sep=2pt, outer sep=0pt, scale=0.75] at (3+\f,\g){}; 
\node[draw,shape=semicircle,rotate=315 ,fill=YellowOrange, anchor=south,inner sep=2pt, outer sep=0pt, scale=0.75] at (3+\f,\g){};
\end{tikzpicture}
\caption[]{Construction of the coefficient tensor $\mathbf{\Xi}$: After the left-orthonormalization of the first $d$ or $p$ cores, respectively, of $\mathbf{\Psi}_{\textrm{cm}/\textrm{fm}} (\mathcal{X})$, the tensor $\mathbf{\Xi}$ can be obtained by contracting the pseudoinverse with the data matrix $\mathcal{Y}$. Since $\mathbf{\Xi}^T$ is a cyclic tensor train, the tensor $\mathbf{\Xi}$ can be expressed as a standard tensor train where we only replace the last core.}
\label{fig: contraction}
\end{figure}

Storing the TT cores of \eqref{eq: TT coordinate major} and \eqref{eq: TT function major} in a sparse format, the memory consumption of both representations can be estimated as $O(p \cdot d \cdot m)$. The memory consumption of the corresponding matricizations would be $O(p^d \cdot m)$ and $O(d^p \cdot m)$, respectively. This is the first main advantage of the proposed decompositions: If the number of snapshots is not too large ($m \ll p^d$ or $m \ll d^p$), we are able to efficiently store all entries of the matrices \eqref{eq: matrix coordinate major} and \eqref{eq: matrix function major} in the TT decompositions \eqref{eq: TT coordinate major} and \eqref{eq: TT function major}, respectively. The second advantage is then the fast computation of the pseudoinverse. The computational effort needed to compute an SVD of the matrices \eqref{eq: matrix coordinate major} and \eqref{eq: matrix function major} and to construct the pseudoinverse given in \eqref{eq: pseudoinverse in full format} can be estimated as $O(p^d \cdot m^2)$ and $O(d^p \cdot m^2)$, respectively. In contrast to that, we only need $O(p \cdot d \cdot m^3)$ steps to compute the pseudoinverse of the TT representations \eqref{eq: TT coordinate major} and \eqref{eq: TT function major}, see Lemma \ref{lemma: pseudoinverse - effort}. On the other hand, if $m \gtrsim p^d$ (or $m \gtrsim d^p$), then the tensor-train based approach is computationally more expensive than the classical methods. However, we still benefit from the reduced storage consumption of the basis tensor using MANDy.

\begin{remark}[Right-orthonormality of the last core]
Note that it is not necessary to right-orthonormalize the last core of $\mathbf{\Psi}_{\textrm{cm}/\textrm{fm}} (\mathcal{X})$ for the construction of its pseudoinverse as shown in Algorithm \ref{alg:Pseudoinverse} since the last core is already right-orthonormal by construction, cf.~\eqref{eq: TT coordinate major} and~\eqref{eq: TT function major}.
\end{remark}

\begin{remark}[Special case of monomial basis functions]
For the special case of monomial basis functions, i.e., for a dictionary given by $\mathbb{D} = \{\psi_1, \dots, \psi_p\} = \{1, x, x^2, \dots , x^{p-1}\}$ with $p = 2^q$, we can decompose the components of the coordinate-major rank-one tensor given in \eqref{eq: psi f,x} even further by writing
\begin{equation*}
\begin{bmatrix}
    \psi_1(x_i) \\ \psi_2(x_i) \\ \vdots \\ \psi_p(x_i) 
\end{bmatrix} =
\begin{bmatrix}
    1 \\ x_i \\ \vdots \\ x_i^{2 ^q -1}
\end{bmatrix} \hat{=}
\begin{bmatrix}
    1 \\ x_i 
\end{bmatrix} \otimes
\begin{bmatrix}
    1 \\ x_i^2 
\end{bmatrix} \otimes
\begin{bmatrix}
    1 \\ x_i^4 
\end{bmatrix} \otimes \dots \otimes
\begin{bmatrix}
    1 \\ x_i^{2^{q-1}} 
\end{bmatrix},
\end{equation*}
which can be easily shown using $\sum_{l=0}^{q-1} 2^l = 2^{q}-1$. We will not make use of this kind of sub-decomposition in this work. Nevertheless, it may be advantageous for the identification of governing equations involving higher-order monomials.
\end{remark}

\begin{remark}[Constraints on the coefficients] Even when using the TT format for the construction of the tensors $\mathbf{\Psi}_\textrm{cm}(\mathcal{X})$ and $\mathbf{\Psi}_\textrm{fm}(\mathcal{X})$, respectively, it is possible to directly include linear (dimensionwise) constraints on the coefficients stored in $\mathbf{\Xi}$. A tensor train that encodes the extra conditions on the coefficients can be seen as an additional snapshot. Together with a corresponding vector appended to $\mathcal{Y}$, we can incorporate the (weighted) contraints.
\end{remark}

\section{Numerical results}\label{sec: results}

In this section, we will present three numerical examples for the application of MANDy, namely Chua's circuit (which was already introduced in Example \ref{ex:SINDy examples}), the Fermi--Pasta--Ulam--Tsingou problem~\cite{FERMI1955}, and the Kuramoto model~\cite{ACEBRON2005}. The numerical experiments have been performed on a Linux machine with 128 GB RAM and an Intel Xeon processor with a clock speed of 3 GHz and 8 cores. The algorithms have been implemented in Python~3.6 and collected in the toolbox Scikit-TT available on GitHub: \url{https://github.com/PGelss/scikit_tt}.

\subsection{Chua's circuit}\label{sec: chua}

As a first experiment, let us consider our example from Section \ref{sec: recovery} again with the intent to illustrate how using tensor products of simple functions allows for the construction of rich sets of basis functions. The first set of basis functions \eqref{chua: first try} used in Example \ref{ex:SINDy examples} corresponds to the coordinate-major rank-one decomposition
\begin{equation*}
    \mathbf{\Psi}_1(X) =
        \begin{bmatrix} 1 \\ x_1 \\ \, x_1^2 \, \end{bmatrix}
        \otimes
        \begin{bmatrix} 1 \\ x_2 \\ \, x_2^2 \, \end{bmatrix}
        \otimes
        \begin{bmatrix} 1 \\ x_3 \\ \, x_3^2 \, \end{bmatrix}
        \in \R^{3 \times 3 \times 3}.
\end{equation*}
As we have seen, these ansatz functions do not lead to a correct recovery of the system dynamics. Instead we have to use a basis set including the absolute values of the coordinates. The basis vector given in \eqref{chua: second try} then corresponds to the function-major rank-one decomposition
\begin{equation}\label{eq: chua function-major}
    \mathbf{\Psi}_2(X) =
        \begin{bmatrix} 1 \\ x_1 \\ x_2 \\ \, x_3 \, \end{bmatrix}
        \otimes
        \begin{bmatrix} 1 \\ \abs{x_1} \\ \abs{x_2} \\ \, \abs{x_3} \, \end{bmatrix}
        \in \R^{4 \times 4}.
\end{equation}
As for the matrix-based approach in Example \ref{ex:SINDy examples}, the approximate coefficient tensor is numerically equal to the exact coefficient tensor (see Appendix \ref{app: chua}) when we apply MANDy with the basis decomposition $\mathbf{\Psi}_2$ and choose the same parameters and number of snapshots. This small example already shows the advantage of using tensor decompositions in terms of storage consumption. For the sparse TT representation of the basis tensor (see \eqref{eq: psi x,f} and \eqref{eq: TT function major}), we only have to store $18000$ entries whereas the matricized counterpart would have $32000$ entries.

\subsection{Fermi--Pasta--Ulam--Tsingou problem}

Let us now consider the Fermi--Pasta--Ulam--Tsingou model, which was introduced by the eponymous physicists and mathematicians in 1953. The underlying model represents a vibrating string by a system of coupled oscillators fixed at the respective end points of the string.

\begin{figure}[htbp]
\centering
\begin{tikzpicture}
\def\d{1.5}
\def\h{1.5}
\def\y{0}
\draw[] (5-3*\d,\y+0.5*\h) -- ++ (0,-\h);
\draw[dashed] (5-2*\d,\y+0.5*\h) -- ++ (0,-\h);
\draw[dashed] (5-1*\d,\y+0.5*\h) -- ++ (0,-\h);
\draw[dashed] (5+1*\d,\y+0.5*\h) -- ++ (0,-\h);
\draw[dashed] (5+2*\d,\y+0.5*\h) -- ++ (0,-\h);
\draw[] (5+3*\d,\y+0.5*\h) -- ++ (0,-\h);
\node[shape=circle,inner sep=0.1cm, outer sep=0, draw, fill=black] at (5-2*\d+0.2,\y) {};
\node[shape=circle,inner sep=0.1cm, outer sep=0, draw, fill=black] at (5-1*\d-0.4,\y) {};
\node[shape=circle,inner sep=0.1cm, outer sep=0, draw, fill=black] at (5+1*\d-0.2,\y) {};
\node[shape=circle,inner sep=0.1cm, outer sep=0, draw, fill=black] at (5+2*\d+0.3,\y) {};

\FPeval{\xa}{5-3*\d}
\FPeval{\xb}{5-2*\d+0.2-0.1}
\draw[] (\xa,\y) -- ++ (0.166*\xb-0.166*\xa,0);
\draw[] (\xa+0.166*\xb-0.166*\xa,\y) -- ++ (0.0833*\xb-0.0833*\xa,0.1*\h);
\draw[] (\xa+0.166*\xb-0.166*\xa+0.0833*\xb-0.0833*\xa,\y+0.1*\h) -- ++ (0.166*\xb-0.166*\xa,-0.2*\h);
\draw[] (\xa+0.333*\xb-0.333*\xa+0.0833*\xb-0.0833*\xa,\y-0.1*\h) -- ++ (0.166*\xb-0.166*\xa,0.2*\h);
\draw[] (\xa+0.5*\xb-0.5*\xa+0.0833*\xb-0.0833*\xa,\y+0.1*\h) -- ++ (0.166*\xb-0.166*\xa,-0.2*\h);
\draw[] (\xa+0.666*\xb-0.666*\xa+0.0833*\xb-0.0833*\xa,\y-0.1*\h) -- ++ (0.0833*\xb-0.0833*\xa,0.1*\h);
\draw[] (\xa+0.666*\xb-0.666*\xa+0.166*\xb-0.166*\xa,\y) -- ++ (0.166*\xb-0.166*\xa,0);

\FPeval{\xa}{5-2*\d+0.2+0.1}
\FPeval{\xb}{5-1*\d-0.4-0.1}
\draw[] (\xa,\y) -- ++ (0.166*\xb-0.166*\xa,0);
\draw[] (\xa+0.166*\xb-0.166*\xa,\y) -- ++ (0.0833*\xb-0.0833*\xa,0.1*\h);
\draw[] (\xa+0.166*\xb-0.166*\xa+0.0833*\xb-0.0833*\xa,\y+0.1*\h) -- ++ (0.166*\xb-0.166*\xa,-0.2*\h);
\draw[] (\xa+0.333*\xb-0.333*\xa+0.0833*\xb-0.0833*\xa,\y-0.1*\h) -- ++ (0.166*\xb-0.166*\xa,0.2*\h);
\draw[] (\xa+0.5*\xb-0.5*\xa+0.0833*\xb-0.0833*\xa,\y+0.1*\h) -- ++ (0.166*\xb-0.166*\xa,-0.2*\h);
\draw[] (\xa+0.666*\xb-0.666*\xa+0.0833*\xb-0.0833*\xa,\y-0.1*\h) -- ++ (0.0833*\xb-0.0833*\xa,0.1*\h);
\draw[] (\xa+0.666*\xb-0.666*\xa+0.166*\xb-0.166*\xa,\y) -- ++ (0.166*\xb-0.166*\xa,0);

\FPeval{\xa}{5-1*\d-0.4+0.1}
\FPeval{\xb}{5-1*\d-0.2+0.1+0.85}
\draw[] (\xa,\y) -- ++ (0.166*\xb-0.166*\xa,0);
\draw[] (\xa+0.166*\xb-0.166*\xa,\y) -- ++ (0.0833*\xb-0.0833*\xa,0.1*\h);
\draw[] (\xa+0.166*\xb-0.166*\xa+0.0833*\xb-0.0833*\xa,\y+0.1*\h) -- ++ (0.166*\xb-0.166*\xa,-0.2*\h);
\draw[] (\xa+0.333*\xb-0.333*\xa+0.0833*\xb-0.0833*\xa,\y-0.1*\h) -- ++ (0.166*\xb-0.166*\xa,0.2*\h);
\draw[] (\xa+0.5*\xb-0.5*\xa+0.0833*\xb-0.0833*\xa,\y+0.1*\h) -- ++ (0.166*\xb-0.166*\xa,-0.2*\h);
\draw[] (\xa+0.666*\xb-0.666*\xa+0.0833*\xb-0.0833*\xa,\y-0.1*\h) -- ++ (0.0833*\xb-0.0833*\xa,0.1*\h);
\draw[] (\xa+0.666*\xb-0.666*\xa+0.166*\xb-0.166*\xa,\y) -- ++ (0.166*\xb-0.166*\xa,0);

\draw[dotted] (5-1*\d-0.2+0.1+0.85+0.1,\y) -- (5+1*\d-0.1-0.1-0.85-0.1,\y);

\FPeval{\xa}{5+1*\d-0.1-0.1-0.85}
\FPeval{\xb}{5+1*\d-0.2-0.1}
\draw[] (\xa,\y) -- ++ (0.166*\xb-0.166*\xa,0);
\draw[] (\xa+0.166*\xb-0.166*\xa,\y) -- ++ (0.0833*\xb-0.0833*\xa,0.1*\h);
\draw[] (\xa+0.166*\xb-0.166*\xa+0.0833*\xb-0.0833*\xa,\y+0.1*\h) -- ++ (0.166*\xb-0.166*\xa,-0.2*\h);
\draw[] (\xa+0.333*\xb-0.333*\xa+0.0833*\xb-0.0833*\xa,\y-0.1*\h) -- ++ (0.166*\xb-0.166*\xa,0.2*\h);
\draw[] (\xa+0.5*\xb-0.5*\xa+0.0833*\xb-0.0833*\xa,\y+0.1*\h) -- ++ (0.166*\xb-0.166*\xa,-0.2*\h);
\draw[] (\xa+0.666*\xb-0.666*\xa+0.0833*\xb-0.0833*\xa,\y-0.1*\h) -- ++ (0.0833*\xb-0.0833*\xa,0.1*\h);
\draw[] (\xa+0.666*\xb-0.666*\xa+0.166*\xb-0.166*\xa,\y) -- ++ (0.166*\xb-0.166*\xa,0);

\FPeval{\xa}{5+1*\d-0.2+0.1}
\FPeval{\xb}{5+2*\d+0.3-0.1}
\draw[] (\xa,\y) -- ++ (0.166*\xb-0.166*\xa,0);
\draw[] (\xa+0.166*\xb-0.166*\xa,\y) -- ++ (0.0833*\xb-0.0833*\xa,0.1*\h);
\draw[] (\xa+0.166*\xb-0.166*\xa+0.0833*\xb-0.0833*\xa,\y+0.1*\h) -- ++ (0.166*\xb-0.166*\xa,-0.2*\h);
\draw[] (\xa+0.333*\xb-0.333*\xa+0.0833*\xb-0.0833*\xa,\y-0.1*\h) -- ++ (0.166*\xb-0.166*\xa,0.2*\h);
\draw[] (\xa+0.5*\xb-0.5*\xa+0.0833*\xb-0.0833*\xa,\y+0.1*\h) -- ++ (0.166*\xb-0.166*\xa,-0.2*\h);
\draw[] (\xa+0.666*\xb-0.666*\xa+0.0833*\xb-0.0833*\xa,\y-0.1*\h) -- ++ (0.0833*\xb-0.0833*\xa,0.1*\h);
\draw[] (\xa+0.666*\xb-0.666*\xa+0.166*\xb-0.166*\xa,\y) -- ++ (0.166*\xb-0.166*\xa,0);

\FPeval{\xa}{5+2*\d+0.15+0.1}
\FPeval{\xb}{5+3*\d}
\draw[] (\xa,\y) -- ++ (0.166*\xb-0.166*\xa,0);
\draw[] (\xa+0.166*\xb-0.166*\xa,\y) -- ++ (0.0833*\xb-0.0833*\xa,0.1*\h);
\draw[] (\xa+0.166*\xb-0.166*\xa+0.0833*\xb-0.0833*\xa,\y+0.1*\h) -- ++ (0.166*\xb-0.166*\xa,-0.2*\h);
\draw[] (\xa+0.333*\xb-0.333*\xa+0.0833*\xb-0.0833*\xa,\y-0.1*\h) -- ++ (0.166*\xb-0.166*\xa,0.2*\h);
\draw[] (\xa+0.5*\xb-0.5*\xa+0.0833*\xb-0.0833*\xa,\y+0.1*\h) -- ++ (0.166*\xb-0.166*\xa,-0.2*\h);
\draw[] (\xa+0.666*\xb-0.666*\xa+0.0833*\xb-0.0833*\xa,\y-0.1*\h) -- ++ (0.0833*\xb-0.0833*\xa,0.1*\h);
\draw[] (\xa+0.666*\xb-0.666*\xa+0.166*\xb-0.166*\xa,\y) -- ++ (0.166*\xb-0.166*\xa,0);

\draw[] (5-2*\d,\y-0.5*\h-0.15) -- (5-2*\d+0.2,\y-0.5*\h-0.15) node[below, midway] {$x_1$};
\draw[] (5-2*\d,\y-0.5*\h-0.15+0.05) -- ++ (0,-0.1);
\draw[] (5-2*\d+0.2,\y-0.5*\h-0.15+0.05) -- ++ (0,-0.1);

\draw[] (5-1*\d,\y-0.5*\h-0.15) -- (5-1*\d-0.4,\y-0.5*\h-0.15) node[below, midway] {$x_2$};
\draw[] (5-1*\d,\y-0.5*\h-0.15+0.05) -- ++ (0,-0.1);
\draw[] (5-1*\d-0.4,\y-0.5*\h-0.15+0.05) -- ++ (0,-0.1);

\draw[] (5+1*\d,\y-0.5*\h-0.15) -- (5+1*\d-0.2,\y-0.5*\h-0.15) node[below, midway] {$x_{d-1}$};
\draw[] (5+1*\d,\y-0.5*\h-0.15+0.05) -- ++ (0,-0.1);
\draw[] (5+1*\d-0.2,\y-0.5*\h-0.15+0.05) -- ++ (0,-0.1);

\draw[] (5+2*\d,\y-0.5*\h-0.15) -- (5+2*\d+0.3,\y-0.5*\h-0.15) node[below, midway] {$x_d$};
\draw[] (5+2*\d,\y-0.5*\h-0.15+0.05) -- ++ (0,-0.1);
\draw[] (5+2*\d+0.3,\y-0.5*\h-0.15+0.05) -- ++ (0,-0.1);
\end{tikzpicture}
\caption{Fermi--Pasta--Ulam--Tsingou model: Representation of a vibrating string by a set of masses coupled by springs.}
\label{fig: fpu model}
\end{figure}
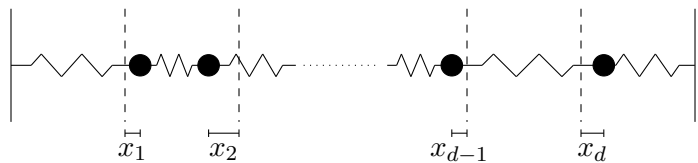

Here, we consider a dynamical system described by a second-order differential equation $\ddot{X}(t) = F(X(t))$ where the right-hand side does not depend on $\dot{X}$. A very large number of problems in astronomy, molecular dynamics, and other areas of physics are of this form, as this type of differential equation is an immediate consequence of Newtonian mechanical laws. Moreover, it makes no difference for our methods whether we have given data measurements $[ X_1, \dots, X_m]$ and $[ \dot{X}_1, \dots, \dot{X}_m]$ or $[ X_1, \dots, X_m]$ and $[ \ddot{X}_1, \dots, \ddot{X}_m]$. Let us consider the model variant with cubic forcing terms, which is of the form
\begin{equation}\label{eq: FPU}
    \ddot{x}_i = (x_{i+1} - 2 x_i + x_{i-1}) + \beta \left[(x_{i+1} - x_i)^3 - (x_i - x_{i-1})^3 \right],
\end{equation}
$ i = 1, \dots, d $, where $ x_i $ represents the displacement of the $ i $th oscillator from its original position, see Figure~\ref{fig: fpu model}. We assume the ends of the chain to be fixed, i.e., $ x_0 = x_{d+1} = 0 $. The parameter $ \beta \in \mathbb{R} $ represents the nonlinear force between the oscillators.

For our first numerical experiment, we consider $d=10$ oscillators. We compare our proposed method with the matrix-based solution of the least-squares problem given in \eqref{eq: min problem in TT}. For this purpose, we choose random displacements in $[-0.1, 0.1]$ for each oscillator and compute the (exact) second derivatives using \eqref{eq: FPU} with $\beta = 0.7$. As basis functions we choose $\mathbb{D} = \{1, x, x^2, x^3\}$. This then results in the representation of $4^{10}$ combinations of function evaluations for every snapshot if we use the coordinate-major ansatz, see \eqref{eq: psi f,x}, such that
\begin{equation*}
  \mathbf{\Psi}_\textrm{cm}(\mathcal{X}) \in \mathbb{R}^{4 \times \dots \times 4 \times m} \quad \textrm{for} \quad \mathcal{X} =
    \begin{bmatrix}
        X_1 & X_2 & \dots & X_m
    \end{bmatrix} \in \mathbb{R}^{10 \times m},
\end{equation*}
where $m$ is the number of considered snapshots.

\begin{figure}[htbp]
    \centering
    \begin{minipage}{0.49\textwidth}
        \centering
        \subfiguretitle{(a)}
        \includegraphics[height=140px]{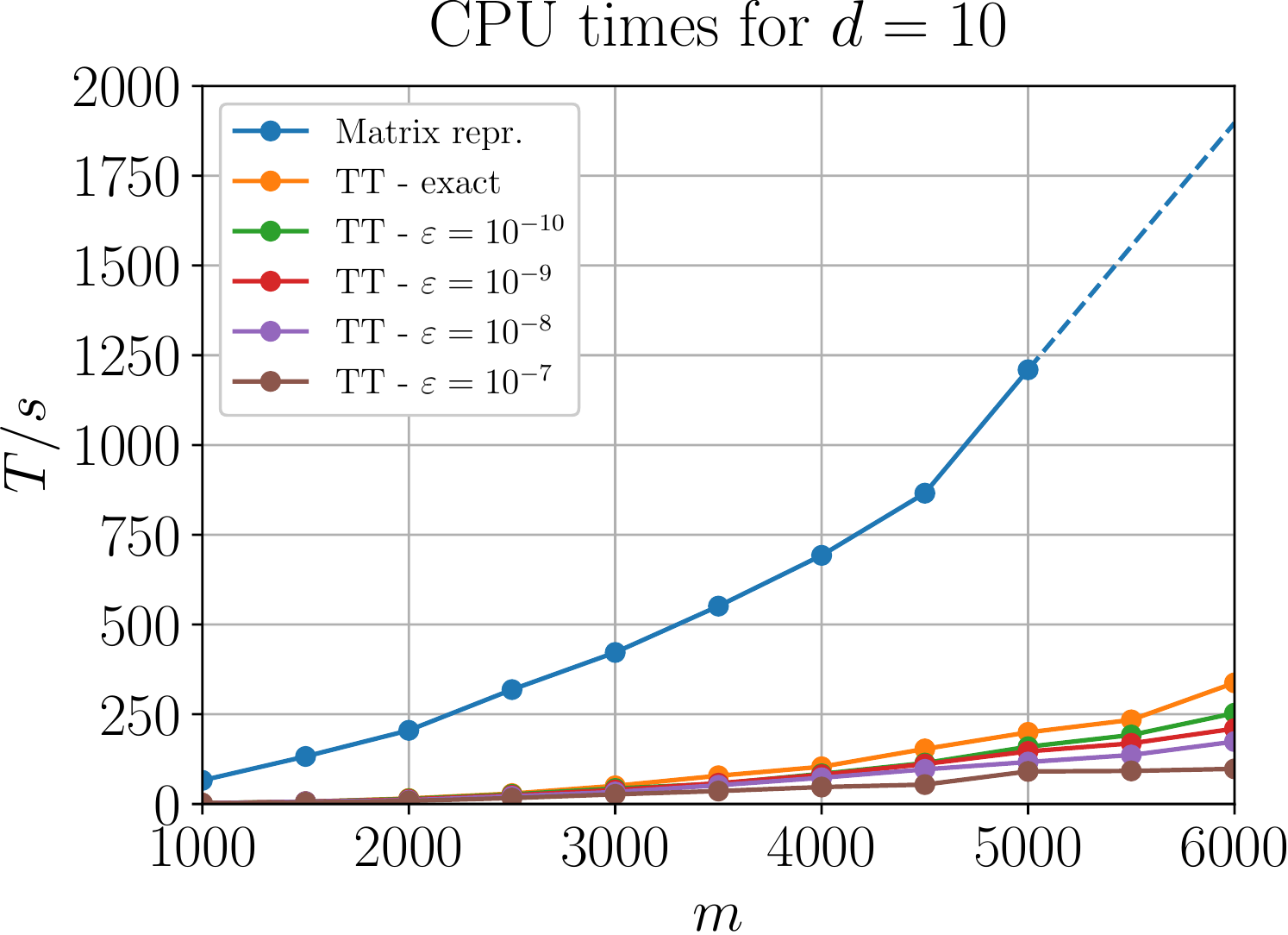}
    \end{minipage}
    \hfill
    \begin{minipage}{0.49\textwidth}
        \centering
        \subfiguretitle{(b)}
        \includegraphics[height=140px]{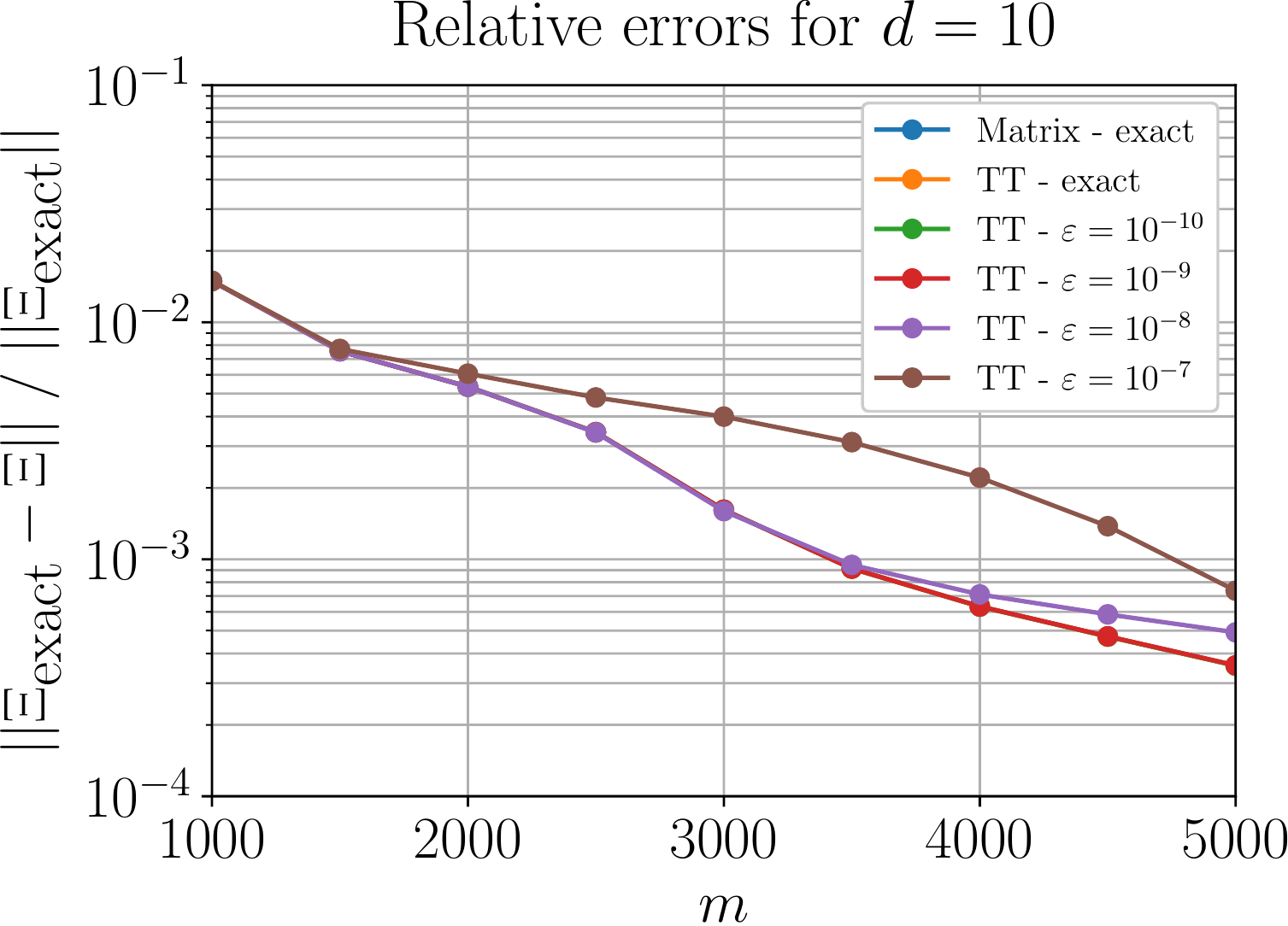}
    \end{minipage}
    \caption{Application of MANDy to the Fermi--Pasta--Ulam--Tsingou problem: (a) CPU times needed for the computation of the coefficient tensor $\mathbf{\Xi}$ and the matricized counterpart $\Xi$, respectively. For $m = 6000$, the CPU time of the matrix-based approach is extrapolated since storing the matrices for calculations with $m > 5000$ would require too much memory. (b) Relative errors between the approximate solutions and the exact solution. Results are shown for the tensor-based approach (using different thresholds for orthonormalizations) as well as for the matrix-based approach. For $\varepsilon \leq 10^{-9}$, the relative errors of MANDy are numerically indistinguishable from those of the SINDy-like approach.}
    \label{fig: fpu 1}
\end{figure}

Figure~\ref{fig: fpu 1} shows the CPU times and relative errors of MANDy for varying $m$. We see that we are able to reduce the time needed for computing the coefficient tensor/matrix significantly -- at least within the considered range of snapshot numbers with $m \ll 4^{10}$. For MANDy, we also included the construction of the tensor train $\mathbf{\Psi}_\textrm{cm}(\mathcal{X})$ in the CPU times, whereas the runtimes for the matrix-based approach only consist of the times needed to solve the minimization problem. Without using truncated SVDs for the orthonormalization procedures, we obtain a speed-up of approximately up to five. The runtimes decrease even further when we set a threshold $\varepsilon > 0$ for the SVDs.

\begin{figure}[htb]
    \centering
    \begin{minipage}{0.95\textwidth}
        \centering
        \includegraphics[height=140px]{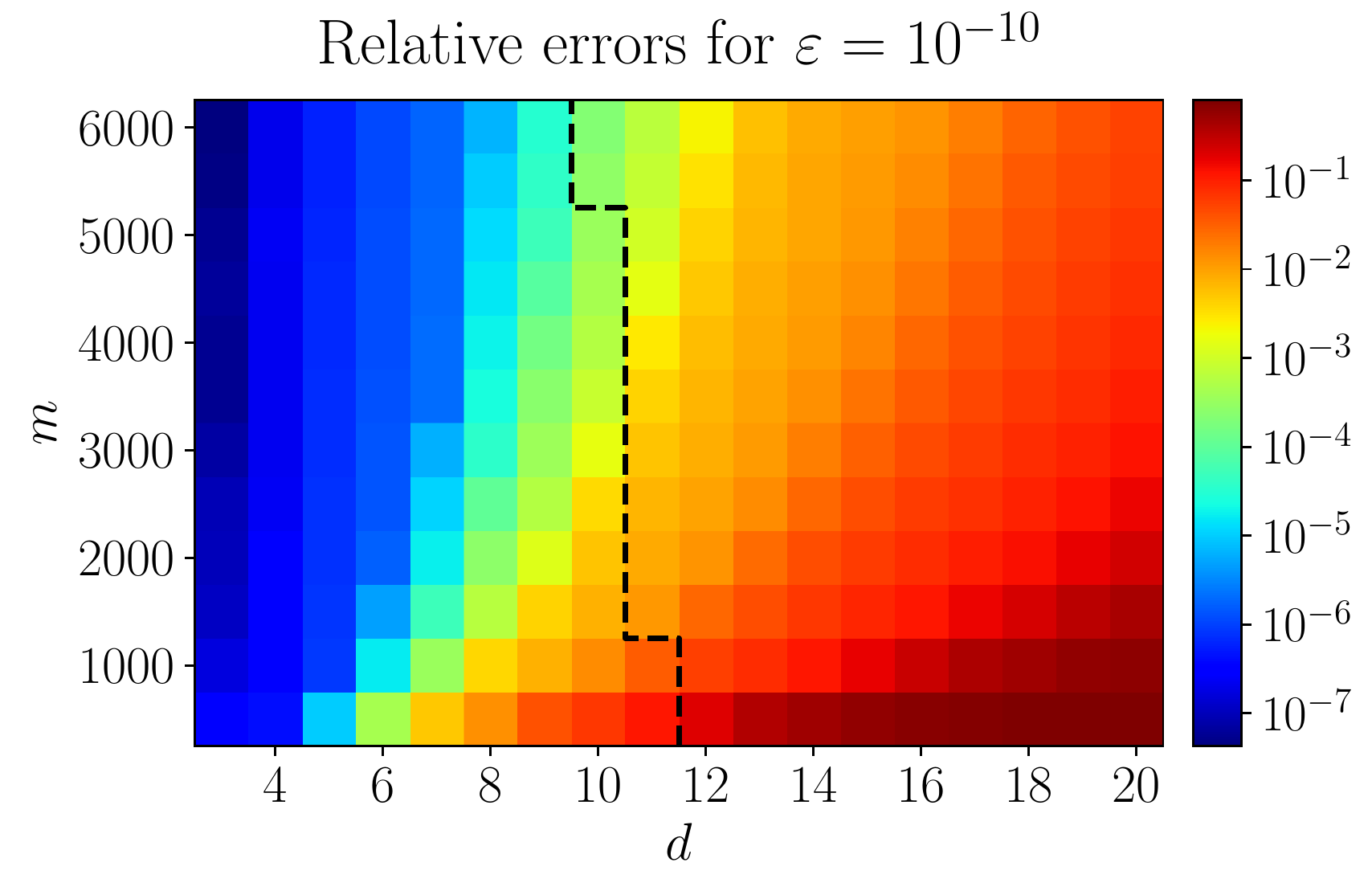}
    \end{minipage}
    \caption{Relative errors for different parameters: The relative errors between the approximate and the exact solutions for the coefficient tensor depending on $d$ and $m$ are shown. All parameter combinations for $m$ and $d$ to the right of the dashed line could only be handled using the tensor-based method MANDy.}
    \label{fig: fpu 2}
\end{figure}

At the same time, we obtain nearly the same relative errors between the approximate and the exact solution, cf.~Appendix~\ref{app: FPU}. As shown in Figure \ref{fig: fpu 1} (b), there is in fact no noticeable difference between the error produced by the matrix-based approach and the tensor-based approach with thresholds of $\varepsilon \leq 10^{-9}$. Moreover, we can consider higher numbers of snapshots and dimensions using MANDy. This is shown more clearly in Figure~\ref{fig: fpu 2}, where we plot the relative errors in dependence of $m$ and $d$. In particular, we were not able to compute a matrix-based solution using standard SINDy for $d \geq 12$ since the basis matrix $\Psi_\textrm{cm}(\mathcal{X})$ simply becomes too large. Using MANDy, we can even approximate the solution with a relative error smaller than $10^{-1}$ for $d = 20$ and $m = 6000$ (which implies a basis tensor $\mathbf{\Psi}_\textrm{cm}(\mathcal{X})$ with $4^{20} \cdot 6000 \approx 6.6 \cdot 10^{15}$ elements).

\subsection{Kuramoto model}

As a third example for the application of MANDy, we consider the Kuramoto model. The model describes the behavior of a large number of coupled oscillators on a circle. First introduced by Yoshiki Kuramoto in 1975, see~\cite{KURAMOTO1975}, it has been studied extensively over the last decades. 

\def\oscillator[#1,#2](#3){
  \FPeval{\xa}{2.25*cos(#3)+#1}
  \FPeval{\ya}{2.25*sin(#3)+#2}
  \draw[dashed] (#1,#2) -- (\xa,\ya);
  \node[shape=circle, draw=black, fill=black, inner sep=0.075cm] at (\xa,\ya) {};
}

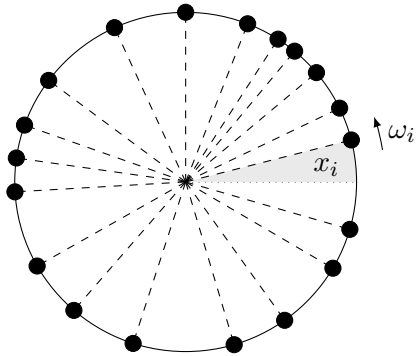
\begin{figure}[htbp]
\centering
\begin{tikzpicture}
\def\x{0}
\def\y{0}

\draw[dotted] (\x,\y) -- ++ (2.25,0);
\draw[fill=Gray!20!white, draw=none] (\x,\y) ++ (14.3239:2.25) arc (14.3239:0:2.25) -- (\x,\y);
\node[] at (\x+1.85,\y+0.195) {$x_i$};
\draw[<-, > = latex] (\x,\y) ++ (14.3239+5:2.625) arc (14.3239+5:14.3239-5:2.625);
\node[] at (\x+2.85,\y+0.6375) {$\omega_i$};

\node[draw=black, shape=circle, inner sep=1.59cm] at (\x,\y) {~};

\oscillator[\x,\y](0.25)
\oscillator[\x,\y](0.45)
\oscillator[\x,\y](0.7)
\oscillator[\x,\y](0.88)
\oscillator[\x,\y](1)
\oscillator[\x,\y](1.2)
\oscillator[\x,\y](1.5708)
\oscillator[\x,\y](2)
\oscillator[\x,\y](2.5)
\oscillator[\x,\y](2.8)
\oscillator[\x,\y](3)
\oscillator[\x,\y](3.2)
\oscillator[\x,\y](3.66)
\oscillator[\x,\y](4)
\oscillator[\x,\y](4.4)
\oscillator[\x,\y](5)
\oscillator[\x,\y](5.33)
\oscillator[\x,\y](5.75)
\oscillator[\x,\y](6)
\end{tikzpicture}
\caption[Kuramoto model]{Kuramoto model: Simulation of coupled oscillators on a ring.}
\label{fig: kuramoto model}
\end{figure}

The governing equations -- including an external forcing, see~\cite{ACEBRON2005} -- can be written as
\begin{equation}\label{eq: kuramoto ODE}
    \frac{\mathrm{d} x_i}{\mathrm{d} t} = \omega_i + \frac{K}{d} \sum_{j=1}^d \sin(x_j - x_i) + h \sin(x_i),\quad i = 1, \dots, d, 
\end{equation}
where $ d $ is the number of oscillators, $ K $ the coupling strength, $h$ the external forcing parameter, $ \omega_i $ the $ i $th natural frequency, and $x_i$ the angular position of the $i$th oscillator, see Figure~\ref{fig: kuramoto model}. Here, we assume an identical coupling between all oscillators as well as a constant forcing parameter. In detail, we consider $d=100$ oscillators and set $K = 2$ and $h = 0.2$. Furthermore, we choose intrinsic frequencies equidistantly distributed in $[-5,5]$.

Now, we intend to identify the governing equations from simulation data. We choose the set of basis functions given by $\mathbb{D} =\{\sin(x), \cos(x)\}$ and use the function-major decomposition \eqref{eq: psi x,f} which leads to tensors of the form
\begin{equation}\label{eq: kuramoto decomposition}
    \mathbf{\Psi}_{\textrm{fm}}(X) = \begin{bmatrix} 1 \\ \sin(x_1) \\ \vdots \\ \sin(x_d) \end{bmatrix} \otimes \begin{bmatrix} 1 \\ \cos(x_1) \\ \vdots \\ \cos(x_d) \end{bmatrix} \in \mathbb{R}^{d +1 \times d+1}.
\end{equation}
Similar to the dictionary used for Chua's circuit, see Section \ref{sec: chua}, we added the basis function $1$ to both cores in order to ensure that also functions of the form $\sin(x_i)$ and $\cos(x_i)$ appear in the large basis set. Note that $ \sin(x \pm y) = \sin(x) \cos(y) \pm \cos (x) \sin(y)$ so that the system can indeed be represented by the chosen basis. We simulate the Kuramoto model from $ t_0 = 0 $ to $ t_1 = 1020 $ using an implementation of the BDF method as described in~\cite{SHAMPINE1997} and take $10$ snapshots within each second of simulation time, i.e., the time points corresponding to the snapshots are $\{0, 0.1 , 0.2, \dots , 1019.9, 1020\}$. That is, we consider $10201$ data points which would for the matrix case mean that the matrix $\Psi_\textrm{fm}(\mathcal{X})$ is square. As the initial distribution, we randomly place the oscillators on the ring. Then, based on the obtained data points $X_1 , \dots , X_{10201}$ and corresponding derivatives $\dot{X}_1 ,  \dots , \dot{X}_{10201}$, we recover the dynamics using MANDy with a threshold of $\varepsilon = 10^{-16}$. As in the previous examples, we are able to compare our result with the exact solution given in Appendix \ref{app: kuramoto}. We repeated the experiment several times and never observed relative errors larger than $10^{-4}$. Applying our recovered dynamics to another randomly chosen initial distribution of the oscillators, we see that we are able to approximate the dynamics of the Kuramoto model accurately within a time interval of nearly up to $100$ seconds, see Figure \ref{fig: kuramoto}. In terms of matrix representations, the reason for the inaccuracy is that the matrix $\Psi_\textrm{fm}(\mathcal{X})$ does not have full rank.

Even if the computational overhead is bigger using MANDy compared to the classical method, we are able to reduce the storage consumption by a factor of more than $50$. Here, a sparse representation of the tensor cores of $\mathbf{\Psi}_\textrm{fm}(\mathcal{X})$ enables us to significantly reduce the memory consumption.

\begin{figure}[htb]
    \centering
    \includegraphics[width=0.9\textwidth]{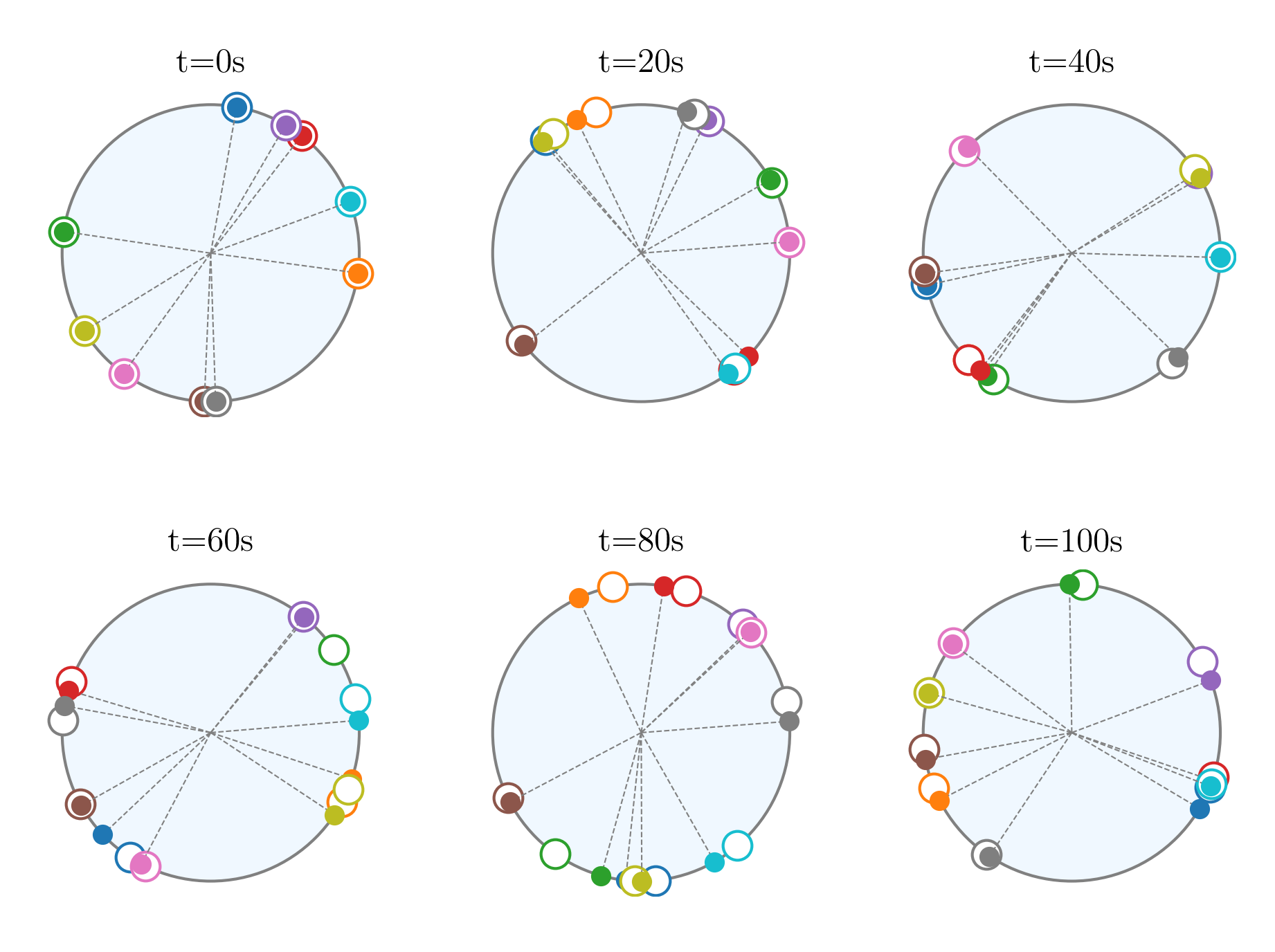}
    \caption[Application of MANDy to the Kuramoto model]{Application of MANDy to the Kuramoto model: The recovered dynamics are applied to another random initial distribution of the oscillators and then compared with the exact dynamics. Here, we only visualize the oscillators with indices $i = 1, 11 , \dots , 91$. The solid dots represent the positions according to the approximated dynamics whereas the circles with the same color represent the positions according to the real dynamics.}
    \label{fig: kuramoto}
\end{figure}

\section{Conclusion and outlook}\label{sec: conclusion}

\subsection{Conclusion} \label{sec: actualconclusion}

In this work, we have proposed an approach for the data-driven recovery of governing equations of nonlinear dynamical systems. The presented approach -- called MANDy, short for \emph{multidimensional approximation of nonlinear dynamical systems} -- combines data-driven methods with tensor-train decompositions. The aim of this approach is to reduce the memory consumption as well as the computational costs significantly and to mitigate the curse of dimensionality. After explaining the SINDy method for identifying governing equations based on a given basis set as well as the TT format, we have first shown how to use tensor products of simple functions for the construction of rich sets of basis functions. We have then explained how to rewrite the data-driven recovery in terms of tensor decompositions and the tensor-based computation of pseudoinverses. The results, which have been illustrated with several examples of dynamical systems, clearly demonstrate that the proposed algorithm needs less computational costs than the classical method if the number of data points is small enough. However, even if a large number of snapshots is needed for the recovery of the governing equations, we still benefit from the reduction of the memory consumption. The work presented in this paper constitutes a first step towards tensor-based techniques for the reconstruction of unknown systems purely from data measurements.

\subsection{Outlook} \label{sec: structure}

The investigations carried out in this work offer a number of further promising considerations. The above numerical experiments show convincingly that one can expect a significant improvement of MANDy over SINDy, not only concerning the memory requirements used in the recovery, but also in the computational effort. This advantage originates from exploiting the underlying data structure by using low-rank tensor decompositions. Key to this is an understanding of the correlations in the matrices that are approximated in the TT format. Appendix \ref{app: correlation} provides some details on the correlation structure required to obtain accurate TT approximations, requiring further research on the local correlation structure.

Other future research will include the investigation of perturbations of the data and their influence on the reconstructed systems. We have observed that MANDy (as well as methods like SINDy) is highly sensitive to noisy data, even if we only add a slight Gaussian noise to the measurements. In this context, the inclusion of approximate derivatives for the recovery will also be of interest. A common way of adding noise is to take data generated by $\dot{X}(t) = F(X(t))+\xi$, where $\xi\in  \mathbb{R}^d$  is a vector capturing errors, either drawn i.i.d.\ reflecting Gaussian noise, or generally with $\| \xi \|_F < \eta $ for some constant $\eta>0$. One might even hope for proving recovery guarantees as they are common in the context of compressed sensing \cite{CompressedSensingIntroRauhut,CompressedSensingIntroBoche}, given certain suitably structured data and noise levels.

Moreover, the construction of elaborate sets of basis functions has to be examined more closely in the future. Here, we used a priori knowledge for the construction of the dictionary, which is generally not known in practice. Additionally, parts of this work -- in particular the proposed basis decompositions -- may be used for the development of other tensor-based counterparts of methods such as EDMD or kernel EDMD. Accepting that the improved performance originates from exploiting the correlation structure, one can hope that other tensor networks are suitable for other kinds of data. Hierarchical tensor formats \cite{HACKBUSCH2009, ARNOLD2013,PhysRevB.90.125154} and entangled project pair states \cite{Orus-AnnPhys-2014,VerstraeteBig} may be suitable as well. The hope is that the present work stimulates such endeavors.

\section*{Acknowledgements}

This research has been partially funded by Deutsche Forschungsgemeinschaft (DFG) through grant CRC 1114 and EI 519/9-1, as well as the Templeton Foundation and the European Research Council (ERC).

{\small
\bibliographystyle{unsrt}
\bibliography{references}
}
\appendix

\section{Exact coefficient tensors}

For any pair of data matrices $\mathcal{X}, \mathcal{Y}$ related by one of the discussed ODE systems, the following tensor product representations of the coefficient tensors $\mathbf{\Xi}$ satisfy $\mathcal{Y} = \mathbf{\Xi}^T \mathbf{\Psi}(\mathcal{X})$ where $\mathbf{\Psi}(\mathcal{X})$ denotes the considered tensor of basis functions applied to the elements of $\mathcal{X}$.

\subsection{Exact solution for Chua's circuit}\label{app: chua}

For the chosen basis functions $\mathbb{D} = \{x, \left| x \right|\}$ and the function-major decomposition given in \eqref{eq: chua function-major}, an exact tensor product representation of the coefficient tensor $\mathbf{\Xi} \in \mathbb{R}^{(d+1) \times (d+1) \times d}$ corresponding to the ODE \eqref{eq: chua ODE} is given by
\begin{equation*}
\begin{split}
 \mathbf{\Xi}~=~ &  
 \left \llbracket \begin{matrix}
  \begin{bmatrix}
   0 \\ -\alpha(1+\delta_1) \\ \alpha \\ 0
  \end{bmatrix} & 
  \begin{bmatrix}
   0 \\ -\alpha\delta_2 \\ 0 \\ 0
  \end{bmatrix} &
  \begin{bmatrix}
   0 \\ 1 \\ -1 \\ 1
  \end{bmatrix} &
  \begin{bmatrix}
   0 \\ 0 \\ -\beta \\ 0
  \end{bmatrix}
 \end{matrix} \right\rrbracket \otimes \left \llbracket \begin{matrix}
  e_1 & 0 & 0  \\
  e_2 & 0 & 0  \\
  0 & e_1 & 0  \\
  0 & 0 & e_1  
 \end{matrix} \right\rrbracket \otimes \left \llbracket \begin{matrix}
  \tilde{e}_1  \\
  \tilde{e}_2  \\
  \tilde{e}_3  
 \end{matrix} \right\rrbracket,
\end{split}
\end{equation*}
with $e_i \in \mathbb{R}^{4}$ and $\tilde{e}_j \in \mathbb{R}^3$ being the $i$th/$j$th unit vector of the standard basis in the respective space.

\subsection{Exact solution for the Fermi--Pasta--Ulam--Tsingou problem}\label{app: FPU}

For the chosen basis functions $\mathbb{D} = \{1 , x, x^2, x^3\}$ and the coordinate-major decomposition given in \eqref{eq: TT coordinate major}, an exact tensor product representation of the coefficient tensor $\mathbf{\Xi} \in \mathbb{R}^{4 \times \dots \times 4 \times d}$ corresponding to the ODE \eqref{eq: FPU} is given by
{ \allowdisplaybreaks
\begin{align*}
 \mathbf{\Xi}~=~ &
 \left \llbracket \begin{matrix}
  -2 e_2 - 2 \beta e_4 & e_1 + 3 \beta e_3 & -3 \beta e_2 & \beta e_1
 \end{matrix} \right\rrbracket \otimes \left \llbracket \begin{matrix}
  e_1 \\ e_2 \\ e_3 \\ e_4
 \end{matrix} \right\rrbracket \otimes
 e_1 \otimes \dots \otimes e_1 \otimes \tilde{e}_1 \\
 & + \sum_{k= 2}^{d-1} e_1 \otimes \dots \otimes e_1 \otimes 
 \left \llbracket \begin{matrix}
  e_1 & e_2 & e_3 & e_4
 \end{matrix} \right\rrbracket \\
 & \qquad \qquad  \otimes \underbrace{\left \llbracket \begin{matrix}
  -2 e_2 - 2 \beta e_4 & e_1 + 3 \beta e_3 & -3 \beta e_2 & \beta e_1 \\ 
  e_1 + 3 \beta e_3 & 0 & 0 & 0 \\
  -3 \beta e_2 & 0 & 0 & 0 \\
  \beta e_1 & 0 & 0 & 0
 \end{matrix} \right\rrbracket}_{k\textrm{th TT core}} \\
 & \qquad \qquad \otimes \left \llbracket \begin{matrix}
  e_1 \\ e_2 \\ e_3 \\ e_4
 \end{matrix} \right\rrbracket  \otimes e_1 \otimes \dots \otimes e_1 \otimes \tilde{e}_k \\
 & + e_1 \otimes \dots \otimes e_1 \otimes 
 \left \llbracket \begin{matrix}
  e_1 & e_2 & e_3 & e_4
 \end{matrix} \right\rrbracket \otimes 
 \left \llbracket \begin{matrix}
  -2 e_2 - 2 \beta e_4  \\
  e_1 + 3 \beta e_3  \\
  -3 \beta e_2  \\
  \beta e_1 
 \end{matrix} \right\rrbracket \otimes \tilde{e}_d,
\end{align*} }
with $e_i \in \mathbb{R}^4$ and $\tilde{e}_j \in \mathbb{R}^d$ being the $i$th/$j$th unit vector of the standard basis in the respective space.

\subsection{Exact solution for the Kuramoto model}\label{app: kuramoto}

For the chosen basis functions $\mathbb{D} = \{\sin(x), \cos(x)\}$ and the function-major decomposition given in \eqref{eq: kuramoto decomposition}, an exact tensor product representation of the coefficient tensor $\mathbf{\Xi} \in \mathbb{R}^{(d+1) \times (d+1) \times d}$ corresponding to the ODE \eqref{eq: kuramoto ODE} is given by
\begin{equation*}
\begin{split}
 \mathbf{\Xi}~=~& e_1 \otimes e_1 \otimes 
 \begin{bmatrix}
  \omega_1 \\ \vdots \\ \omega_d
 \end{bmatrix} + 
 h \cdot \sum_{k=1}^d e_{k+1} \otimes e_1 \otimes \tilde{e}_k \\[0.2cm]
 & + \frac{K}{d} \cdot \sum_{k=1}^d \left( \sum_{\substack{l=1\\l \neq k}}^d e_{k+1}\right) \otimes e_{k+1} \otimes \tilde{e}_k
 - \frac{K}{d} \cdot \sum_{k=1}^d  e_{k+1} \otimes \left( \sum_{\substack{l=1\\l \neq k}}^d e_{k+1}\right) \otimes \tilde{e}_k,
\end{split}
\end{equation*}
with $e_i \in \mathbb{R}^{d+1}$ and $\tilde{e}_j \in \mathbb{R}^d$ being the $i$th/$j$th unit vector of the standard basis in the respective space.

\subsection{Role of entropies, locality and correlations}\label{app: correlation}

Since the considerations in this work rely strongly on the TT format, we briefly note what type of correlation structure is required to arrive at an efficient TT approximation of a given vector, cf.~\cite{PhysRevB.73.094423,AreaReview}. Let us discuss to what extent this format can approximate a given unstructured vector $X \in \mathbb{R}^d$ (as they are arising in SINDy). It is known, see~\cite{PhysRevB.73.094423}, that for any vector $X$, there exists a tensor train $ \mathbf{X}$ with TT ranks $r_1,\dots, r_{d-1}=r$ ($r_0 = r_d =1$) for which the standard Euclidean vector norm satisfies
\begin{equation*}
    \| \mathbf{X}  - X\|^2 \leq 2  \sum_{l=0}^{d-1} \epsilon_l (r),
\end{equation*}
with
\begin{equation*}
    \epsilon_l (r) := \sum_{i=r+1}^{N_i} \mu_l^i,
\end{equation*}
where $\{\mu_l^i: i=1, \dots, n_i\}$ are the eigenvalues of the partial traces over indices labeled $l+1,\dots, d$ of the real symmetric matrices
\begin{equation*}
    R_l:= {\textrm{tr}}_{l+1,\dots, d}(X X^T).
\end{equation*}
In this sense, the global quality of the approximation of $X$ by a suitable tensor train $ \mathbf{X}$ can be related to thresholding errors. These errors are expected to be small when correlations are local and short-ranged. These correlations follow what is called an area law \cite{AreaReview}. The above bounds can then be brought into contact with entropic correlation measures, i.e.,
\begin{equation*}
    \log(\epsilon_l (r) ) \leq S^{1/2} (R_l) - \log(2r),
\end{equation*}
where $S^{1/2}(R_l) = 2 \log {\textrm{tr}} \left(R_l^{1/2}\right)$ denotes the $1/2$-Renyi entropy which appropriately captures correlations. If the correlations in the vectors are short-ranged in this
sense, an efficient TT approximation is possible.

\end{document}